\documentclass{article}
\usepackage[top=1in, bottom=1in, left=1in, right=1in]{geometry}
\usepackage[table]{xcolor}
\usepackage{geometry}
\usepackage{amsthm}
\usepackage{longtable}
\usepackage{amsmath}
\usepackage{graphicx}
\usepackage{float}

\usepackage{tikz}
\usepackage{amssymb}
\usepackage{enumitem}

\usepackage[lined,boxed,ruled]{algorithm2e}
\usepackage{multirow}
\usepackage[colorlinks,
          linkcolor = teal,		
          anchorcolor = blue,	
          urlcolor = blue,		
          citecolor = teal,
           ]{hyperref}
\usepackage{appendix}
\usepackage{caption}
\usepackage{booktabs}       
\usepackage[figuresright]{rotating} 
\usepackage{makecell} 
\usepackage{multirow}       

\usepackage{xparse}

\usepackage[numbers]{natbib}

\newtheorem{definition}{Definition}[section]
\newtheorem{lemma}{Lemma}[section]

\newtheorem{remark}{Remark}[section]

\newtheorem{thom}{Theorem}[section]

\newcommand\msc[1]{\textbf{Mathematical Subject Classifications}: #1}

\allowdisplaybreaks[4] 

\begin{document}
\title{Relaxed Proximal Point Algorithm: Tight Complexity Bounds and Acceleration without Momentum}

\author{
Bofan Wang\footnotemark[1], \
Shiqian Ma\footnotemark[2],\hspace{0.03cm}\footnotemark[3] \
Junfeng Yang\footnotemark[1],\hspace{0.03cm}\footnotemark[3] \
Danqing Zhou\footnotemark[1]}

\footnotetext[1]{School of Mathematics, Nanjing University, Nanjing,  P. R. China. Research supported by National Natural Science Foundation of China (12371301, 12431011). Email: wangbf@smail.nju.edu.cn,  zhoudanqing@smail.nju.edu.cn.}

\footnotetext[2]{Department of Computational Applied Mathematics and Operations Research, Rice University, Houston, TX, USA.}

\footnotetext[3]{Corresponding authors: Shiqian Ma (sqma@rice.edu) and Junfeng Yang (jfyang@nju.edu.cn).}

\date{}

\maketitle
\begin{abstract}
In this paper, we focus on the \emph{relaxed proximal point algorithm} (RPPA) for solving convex (possibly nonsmooth) optimization problems. 
We conduct a comprehensive study on three types of relaxation schedules: (i) constant schedule with relaxation parameter $\alpha_k\equiv \alpha \in (0, \sqrt{2}]$, (ii) the dynamic schedule put forward by Teboulle and Vaisbourd \cite{TV23}, and (iii) the silver stepsize schedule proposed by Altschuler and Parrilo \cite{AP23b}. The latter two schedules were initially investigated for the \emph{gradient descent} (GD) method and are extended to the RPPA in this paper.
For type (i),  we establish tight non-ergodic $O(1/N)$ convergence rate results measured by function value residual and subgradient norm, where $N$ denotes the iteration counter.
For type (ii), we establish a convergence rate that is tight and approximately $\sqrt{2}$ times better than the constant schedule of type (i). 
For type (iii), aside from the original silver stepsize schedule put forward by Altschuler and Parrilo, we propose two new modified silver stepsize schedules, and 
for all the three silver stepsize schedules, $O(1 / N^{1.2716})$ accelerated convergence rate results with respect to three different performance metrics are established.
Furthermore, our research affirms the conjecture in \cite[Conjecture 3.2]{LG24} on GD method with the original silver stepsize schedule.
\end{abstract}
 
\noindent \msc 15A18, 15A69, 65F15, 90C33

\section{Introduction}

In this paper, we consider the following general convex optimization problem
\begin{equation}
  \begin{array}{ll}
    \underset{x\in\mathbb{R}^{d}}{\mbox{minimize}} & f(x),\end{array}  \tag{\ensuremath{\mathcal{Q}}}\label{eq:merely_cvx_problem}
\end{equation}
where $f \colon \mathbb{R}^{d}\to \mathbb{R}\cup\{\infty\}$ is an extended-real-valued proper, closed, and convex (possibly nonsmooth) function.
Denote the optimal solution set of \eqref{eq:merely_cvx_problem} by $\mathcal{X}^{\star}$, which is assumed to be nonempty. Throughout this paper, we let  $x^{\star}\in \mathcal{X}^{\star}$ be arbitrarily fixed. 
The \emph{proximal mapping} \cite{Mor65} of $f$ is defined by
\begin{equation} \label{def-prox}
    \mathrm{Prox}_{\lambda f}(x) = \arg \min_{y \in \mathbb{R}^d} \left\{f(y)+\frac{1}{2\lambda}\|y-x\|^2\right\}, \quad \forall x\in\mathbb{R}^d, \; \forall \lambda > 0.
\end{equation}
A classic approach for solving \eqref{eq:merely_cvx_problem} is the \emph{proximal point algorithm} (PPA), which, starting at $x^0 \in \mathbb{R}^{d}$, iterates for $k\geq 0$ as $x^{k+1} = \mathrm{Prox}_{\lambda f}(x^k)$.
PPA was first introduced in \cite{Mar70} and was further studied by Rockafellar \cite{Roc76a, Roc76b}. 

Convergence rate results of PPA with varying proximal parameters $\{ \lambda_k \}_{k=0}^{N-1}$ for solving \eqref{eq:merely_cvx_problem} was first established by G{\"u}ler  \cite{Gul91}, where $O(1/ \sum_{k=0}^{N-1} \lambda_k)$ complexity bounds for both function value residual and subgradient norm were derived.
Recently, the results derived by \cite{Gul91} were tightened by improving a constant factor in \cite{THG17a}.

A common technique for accelerating convergence is relaxation, which originates from methods like successive over-relaxation for solving linear systems \cite{Sou46, Var62} and Krasnosel'skiǐ-Mann (KM) iteration for fixed-point algorithms \cite{Kra55, Man53}.
Given a positive integer $N\in\mathbb{N}_+$,  an initialization point $x^0 \in \mathbb{R}^{d}$ and $N$ relaxation parameters $\{\alpha_0,\alpha_1, \ldots, \alpha_{N-1}\}\subseteq  \mathbb{R}_{++}$,  the $N$-step \emph{relaxed proximal point algorithm} (RPPA) for solving \eqref{eq:merely_cvx_problem} 
is summarized in Algorithm \ref{rppa}. The sequence of relaxation parameters, denoted by vector $\boldsymbol{\alpha} = (\alpha_0,\alpha_1, \ldots ,\alpha_{N-1})  \in \mathbb{R}^N_{++}$, is also referred to as the relaxation schedule. 
It has been widely recognized empirically that relaxation, especially over-relaxation, i.e., $\alpha_i>1$, can enhance the numerical performance of fixed-point algorithms and optimization methods, see, e.g., \cite{EB92,Boy+11,CGH14}.

\begin{algorithm}[!htb]
\caption{Relaxed Proximal Point Algorithm}\label{rppa}
  \SetKwInOut{Input}{input}
  \SetKwInOut{Output}{output}
  \Input{$x^{0} \in \mathbb{R}^d$, $\lambda>0$, $N\in \mathbb{N}_+$, $\boldsymbol{\alpha} \in  \mathbb{R}^N_{++}$}
  \For{$k=0,1,2,\ldots,N-1$}{
  $z^{k} = \mathrm{Prox}_{\lambda f}(x^{k})$ \dotfill (Proximal step)\label{alg:prox-step}\\
  $x^{k+1} = x^{k}+ \alpha_{k}(z^{k}-x^{k})$ \dotfill (Relaxation step)\label{alg:relax-step}
  }
  \Output{$z^{N} = \mathrm{Prox}_{\lambda f}(x^{N})$}
\end{algorithm}

Various studies have analyzed the convergence of RPPA in different settings.
For instance, it has been applied to solving variational inequality problems in \cite{LY13, GY22}, where ergodic sublinear convergence rates were established.  
However, the performance measures used in these works cannot be applied to our setting.
In \cite{CY14, GHY14}, RPPA was analyzed for solving monotone inclusion and mixed variational inequality problems.
While the performance measures used in these works can be adapted to our setting, only ergodic convergence rates were provided.
Recently, with the fixed point residual being used as the performance measure, tight non-ergodic convergence rate results of the KM iteration were established in \cite{Lie18}, and these results also apply to RPPA.
Detailed convergence bound results of RPPA under different measures for different stepsize schedules, along with the corresponding references, are summarized in Table \ref{tab:RPPA-convergence-compare}.

Despite the various studies on RPPA, current research has several limitations.
As far as we know, there has been no analysis regarding the tight convergence rate of RPPA with a constant relaxation schedule for solving \eqref{eq:merely_cvx_problem}.
Moreover, existing studies fail to demonstrate a clear advantage of RPPA over PPA in terms of the order of convergence rate.
In this paper, we conduct a theoretical examination of how different relaxation schedules impact the worst-case guarantees of PPA. We will study existing relaxation schedules and also propose new ones.
In particular, we establish tight convergence rate results for RPPA and show that RPPA can achieve improved complexity bounds compared to PPA.
As byproducts, we also solve some conjectures in the literature by providing their formal proofs.
A majority of the main results derived in this paper are summarized in Table \ref{tab:RPPA-convergence-compare}.

\begin{sidewaystable}
  \centering
  \begin{tabular}{@{}lllll@{}}
  \toprule
    Measure
    & Stepsize Schedule $\boldsymbol{\alpha}$
    & $N$ arbitrary?
    & Upper Bound
    & Lower Bound \\
  \midrule
    \multirow{4}{*}{$\frac{f(z^N)-f(x^\star)}{\|x^0-x^\star\|^2}$}
    & $\alpha_k\equiv 1$ (reduce to PPA)
    & Yes
    & $\frac{1}{1+N} \frac{1}{4\lambda}$ \cite[Thm. 4.1]{THG17a}
    & \multirow{4}{*}{\begin{tabular}[c]{@{}c@{}}
        $\frac{1}{1+\sum_{k=0}^{N-1}\alpha_k} \frac{1}{4\lambda}$\\ 
        $[\textbf{Ours, Lem. \ref{RPPA_general_lower_bound} (ii)}]$
        \end{tabular}} \\
    & $\alpha_k\equiv \alpha \in (0,\sqrt{2}]$
    & Yes
    & $\frac{1}{1+N\alpha} \frac{1}{4\lambda}$  $[\textbf{Ours, Thm. \ref{thom-tight-analysis}}]$
    & \\ 
  \cmidrule(lr){2-4}
    & Teboulle and Vaisbourd's schedule \eqref{dynamic_stepsize} \cite{TV23}
    & Yes
    & $\frac{1}{1+\sum_{k=0}^{N-1}\alpha_k}\frac{1}{4\lambda}\sim\frac{1}{1+2N}\frac{1}{4\lambda}$ $[\textbf{Ours, Thm. \ref{RPPA_short_dynamic}}]$
    & \\ 
  \cmidrule(lr){2-4}
    & Right silver schedule $\overline{\pi}^{(m)}$ \eqref{def:right_silver} $[\textbf{Ours}]$ 
    & $N=2^{m-1}$
    & $\frac{1}{1+\sum_{k=0}^{N-1}\alpha_k}\frac{1}{4\lambda}\sim \frac{1}{ N^{\log_2 \rho}}\frac{1}{4\lambda}$ $[\textbf{Ours, Thm. \ref{thm:objval-vs-ptnorm}}]$
    & \\
  \midrule
    \multirow{2}{*}{$\frac{\|\nabla f^\lambda(x^N)\|^2}{f(x^0)-f(x^\star)}$} 
    & \multirow{2}{*}{Left silver schedule $\underline{\pi}^{(m)}$ \eqref{def:left_silver} $[\textbf{Ours}]$} 
    & \multirow{2}{*}{$N=2^m$} 
    & \multirow{2}{*}{$\frac{1}{1+\sum_{k=0}^{N-1}\alpha_k}\frac{1}{\lambda}\sim \frac{1}{ N^{\log_2 \rho}}\frac{1}{\lambda}$ $[\textbf{Ours, Thm. \ref{thm:subgdnorm-vs-objval}}]$}
    & \multirow{2}{*}{\begin{tabular}[c]{@{}c@{}}
        $\frac{1}{1+\sum_{k=0}^{N-1}\alpha_k}\frac{1}{\lambda}$\\
        $[\textbf{Ours, Lem. \ref{RPPA_general_lower_bound} (iii)}]$
      \end{tabular}} \\
    &
    &
    &
    & \\ 
  \midrule
    \multirow{4}{*}{$\frac{\|\nabla f^\lambda(x^N)\|}{\|x^0-x^\star\|}$}
    & $\alpha_k\equiv \alpha \in [1,2)$
    & Yes
    & \begin{tabular}[c]{@{}ll@{}}
      $\frac{1}{\sqrt{N+1}} \left( \frac{N}{N+1} \right)^{N /2} \frac{1}{\lambda \sqrt{\alpha (2-\alpha)}}$, &   if $1\le \alpha\le 1+\sqrt{\frac{N}{N+1}}$ \\
      $(\alpha-1)^N \frac{1}{\lambda}$, &  if $1+\sqrt{\frac{N}{N+1}}<\alpha < 2$ \\ 
      \cite[Thm. 4.9]{Lie18} & 
      \end{tabular} 
    & \multirow{4}{*}{\begin{tabular}[c]{@{}c@{}}
      $\frac{1}{1+\sum_{k=0}^{N-1}\alpha_k}\frac{1}{\lambda}$\\ 
      $[\textbf{Ours, Lem. \ref{RPPA_general_lower_bound} (i)}]$
      \end{tabular}}  \\
  \cmidrule(lr){2-4}
    & $\alpha_k \equiv  1$ (reduce to PPA)
    & Yes
    & $\frac{1}{1+N} \frac{1}{\lambda}$ \cite[Thm. 2.2]{Gul91}
    & \\
  \cmidrule(lr){2-4}
    & $\alpha_k\equiv \alpha \in (0,\sqrt{2}]$
    & Yes
    & 
      $\frac{1}{1+\alpha N} \frac{1}{\lambda}$
      $[\textbf{Ours, Rmk. \ref{rmk:constant_step_subgdnorm}}]$
    & \\ 
  \cmidrule(lr){2-4}
    & Silver schedule $\pi^{(m)}$ \eqref{def:silver_step} \cite{AP23b}
    & $N=2^m-1$
    & $\frac{1}{1+\sum_{k=0}^{N-1}\alpha_k}\frac{1}{\lambda} \sim \frac{1}{ N^{\log_2 \rho}}\frac{1}{\lambda}$  $[\textbf{Ours, Thm. \ref{thm:RPPA_silver_tight_basic}}]$
    & \\ 
  \bottomrule
  \end{tabular}
  \caption{{Convergence bound results of RPPA with different stepsize schedules under different measures for solving \eqref{eq:merely_cvx_problem}.
   $\rho=1+\sqrt{2}\approx 2.414$ is the silver ratio.
  \label{tab:RPPA-convergence-compare}}}
\end{sidewaystable}

\subsection{Connection with gradient descent methods} 
Our study is closely related to the recent advances in the analysis of \emph{gradient descent} (GD) methods  \cite{DT14,TV23,AP23b,GSW24}. This is because RPPA can be viewed as applying GD method to the Moreau envelope of $f$, which is defined as 
$$
f^\lambda (x) := \min_{u \in  \mathbb{R}^{d}} \left\{ f(u) + \frac{1}{2\lambda} \left\| x-u \right\|^{2} \right\}, \quad x\in\mathbb{R}^d, \; \lambda>0.
$$
It is well known that $f^\lambda$ is convex, $(1 / \lambda)$-smooth, i.e., it is differentiable and its gradient is $(1/\lambda)$-Lipschitz continuous, and furthermore, it shares the same set of minimizers as $f$, see, e.g., \cite[Sec. 6.7]{Bec17}.
In particular, for $k\in \{\star, 0,1,2,\ldots, N-1\}$, the sequence generated by Algorithm \ref{rppa} together with $z^\star := \mathrm{Prox}_{\lambda f}(x^\star) = x^\star$ satisfy
\begin{align}
  \frac{x^k - z^{k}}{\lambda}  &= \frac{x^k - \mathrm{Prox}_{\lambda f}(x^k)}{\lambda} = \nabla f^{\lambda}(x^k) \in  \partial f(z^{k}), \label{eq:gd_mor_envlp}\\
  f^{\lambda}(x^k) &= f(z^k) + \frac{1}{2\lambda} \left\| x^k- z^k \right\|^{2} = f(z^k) + \frac{\lambda}{2} \left\| \nabla f^{\lambda}(x^k) \right\|^{2}. \label{eq:fval_mor_envlp}
\end{align}
As a result, the iterative scheme of RPPA can be equivalently expressed as 
\begin{equation}
  x^{k+1} = x^k + \alpha_k(\mathrm{Prox}_{\lambda f}(x^k) - x^k)  \equiv   x^k - \alpha_k \lambda \nabla f^{\lambda} (x^k), \quad k=0,1, \ldots ,N-1 ,\label{alg:R-PPA-compact-form}
\end{equation}
which is precisely the $N$-step GD method  
applied to the $(1 / \lambda)$-smooth function $f^{\lambda}$ with stepsize schedule $\boldsymbol{\alpha}=(\alpha_0,\alpha_1, \ldots ,\alpha_{N-1})$ normalized by the Lipschitz constant $(1/\lambda)$.
For this reason, we will also refer to the relaxation schedule $\boldsymbol{\alpha}$ of RPPA as the stepsize schedule.

Given the connection expressed in \eqref{alg:R-PPA-compact-form} between RPPA and GD methods, it appears that the convergence results of GD methods can be extended to RPPA in a straightforward manner. However, if extended in this way, the resulting results are characterized by the Moreau envelope $f^\lambda$, whereas our primary interest lies in $f$ itself.
Moreover, the results obtained in this manner are usually worse than the tight complexity bound that we desire.
Next, we will review the key results from the GD literature.  
This will help place our contributions in context.


\subsection{Prior arts on stepsize design for gradient descent methods} \label{pa}

Let $x^0 \in \mathbb{R}^{d}$ be a given starting point.
Consider minimizing an $L$-smooth convex function $h$ via the $N$-step GD method with stepsize schedule $\boldsymbol{\alpha}=(\alpha_0,\alpha_1, \ldots ,\alpha_{N-1}) \in \mathbb{R}^{N}_{++}$ normalized by $L$, that is
\begin{equation} \label{alg:N-step-GD}
  x^{k+1} = x^k - \frac{\alpha_k}{L} \nabla h(x^k), \quad k=0,1, \ldots ,N-1.
\end{equation}
Traditional analysis of GD is confined to the scenario of ``short stepsizes". In this case, the (normalized) stepsizes $\alpha_{k}$ are chosen in $(0, 2)$, as seen \cite{DT14, Kim+23,THG17b, TV23}.
This strategy ensures a decrease in the function value at each iteration.
For constant stepsizes $\alpha_k\equiv \alpha \in (0,1]$, tight convergence rate results have been established under three optimality measures, which are 
$\left\| \nabla h(x^{N}) \right\| / \left\| x^{0} - x^{\star} \right\|$ in \cite{TV23}, 
$(h(x^N)-h(x^{\star})) / \left\| x^{0}-x^{\star} \right\|^{2}$ in \cite{DT14, TV23} and
$\left\| \nabla h(x^N) \right\|^{2} / (h(x^{0})-h(x^{\star}))$ in \cite{Kim+23}.
Very recently, an exact full-range analysis of GD method for any constant stepsizes $\alpha_k\equiv \alpha \in (0,2)$ was conducted in \cite{RGP24}.
In \cite{TV23}, a dynamic stepsize schedule with $\alpha_k\rightarrow2$ was proposed. This schedule improves the complexity bound of the GD method with $\alpha_k\equiv1$ by a factor of $2$ asymptotically for the measure $(h(x^{N})-h(x^{\star})) / \left\| x^0-x^{\star} \right\|^{2}$.
Similarly, another dynamic schedule was introduced in \cite{RGP24}, which achieved an accelerated rate under the measure $\left\| \nabla h(x^N) \right\|^{2} / (h(x^{0})-h(x^{\star}))$.
In this ``short stepsize'' case, the best possible convergence rate without momentum is $O(1 / N)$, see \cite[Sec. 6.1.1]{DVR23}.

Recent research has delved into “long stepsizes”, where the (normalized) stepsizes $\alpha_k$ can surpass the bound of 2. Although this approach may result in an increase in function value at some iterations, a meticulous design of the stepsizes can potentially enhance the final complexity bound.
Altschuler's master's thesis \cite{Alt18} first identified optimal stepsize schedules for strongly convex problems with $N = 2, 3$.
Das Gupta et al. proposed in \cite{DVR23} the use of a branch-and-bound technique to solve the nonconvex performance estimation problem. This approach, referred to as BnB-PEP in the rest of this paper, can numerically find the optimal stepsize schedules for any $N\geq1$.
However, due to the nonconvexity of the stepsize optimization problem, only optimal schedules for $N\leq50$ were provided as only small-scale problems can be solved by the branch-and-bound approach due to their nonconvex nature.
Grimmer \cite{Gri23} achieved asymptotic convergence rates by cycling through finite “straightforward” schedules, thereby improving the classical convergence rates by a constant factor.

All the works reviewed above maintain the convergence rate of $O(1 / N)$ for the GD method without altering the order.  
However, recent studies \cite{AP23a,AP23b,GSW23, GSW24} have developed stepsize sequences that boost the convergence rate of GD method to a higher order. 
Motivated by the numerically optimal step sizes in \cite{DVR23}, Grimmer et al. in \cite{GSW23} proposed an infinite “straightforward” stepsize schedule, which achieves an $O(1 /N^{1.0564})$ rate in the non-strongly convex setting.
In \cite{AP23a}, Altschuler and Parrilo introduced the \emph{silver stepsize schedule} for the GD method in the $\mu$-strongly convex setting, which is named after its use of the silver ratio $\rho:=1+\sqrt{2}$ in the stepsizes.
In the limit as $\mu$ approaches $0$, an accelerated $O(1 / N^{1.2716})$ rate for the function value residual in the non-strongly convex case was derived in \cite{AP23b}.
Later,  Grimmer et al. \cite{GSW24} proposed a ``right-side heavy'' stepsize schedule based on the silver stepsize schedule, which improves the original result by a constant factor of approximately $2.3244$.
Furthermore, motivated by the H-duality of gradient-based algorithms as described in \cite{Kim+23}, Grimmer et al. in \cite{GSW24} also introduced a ``left-side heavy'' stepsize schedule. This schedule mirrors the “right-side heavy” stepsize schedule and achieves an $O(1 / N^{1.2716})$ rate for the gradient norm.

The above-reviewed stepsize schedules, along with their corresponding convergence rate results and references, are summarized in Table \ref{tab:GD-convergence-compare}.

\begin{sidewaystable}
  \centering
  \begin{tabular}{@{}lllll@{}}
  \toprule
    Measure 
    & Stepsize Schedule $\boldsymbol{\alpha}$
    & $N$ arbitrary?
    & Upper Bound
    & Lower Bound\\ 
  \midrule
    \multirow{5}{*}{$\frac{h(x^N)-h(x^\star)}{\|x^0-x^\star\|^2}$} 
    & $\alpha_k\equiv \alpha \in (0,1]$
    & Yes
    & $\frac{L}{2+4N\alpha}$ \cite{DT14, TV23}
    & \multirow{5}{*}{\makecell[c]{$\frac{L}{2+4\sum_{k=0}^{N-1}\alpha_k}$ \\ 
      \cite[Lem. 3.1]{LG24}}}  \\ 
  \cmidrule(lr){2-4}
    & Teboulle and Vaisbourd's schedule \eqref{dynamic_stepsize} \cite{TV23}
    & Yes
    & $\frac{L}{2+4\sum_{k=0}^{N-1}\alpha_k} \sim \frac{L}{2+8N}$ \cite{TV23}
    & \\ 
  \cmidrule(lr){2-4}
    & \multirow{2}{*}{Silver schedule $\pi^{(m)}$ \eqref{def:silver_step} \cite{AP23b}}
    & \multirow{2}{*}{$N=2^m-1$} & $\frac{L}{1+\sqrt{4\rho^{2m}-3}}\sim \frac{L}{2 N^{\log_2 \rho}}$ \cite{AP23b}
    &  \\
    &
    &
    & $\frac{L}{2+4\sum_{k=0}^{N-1}\alpha_k}=\frac{L}{4\rho^m-2} \sim \frac{L}{4 N^{\log_2 \rho}}$ [\textbf{Ours, Thm. \ref{thm:silver_gd_tight}}]
    &  \\
  \cmidrule(lr){2-4}
    & Right-side heavy schedule $\mathfrak{h}^{(m)}_{\rm right}$ \cite{GSW24}
    & $N=2^m-1$
    & $\frac{L}{2+4\sum_{k=0}^{N-1}\alpha_k}\sim \frac{1}{2(\rho-1)(1+\rho^{-1/2})} \frac{L}{ N^{\log_2 \rho}}$ \cite{GSW24}
    &  \\ 
  \midrule
    \multirow{4}{*}{$\frac{\|\nabla h(x^N)\|^2}{h(x^0)-h(x^\star)}$}
    & $\alpha_k\equiv \alpha \in (0,1]$
    & Yes
    & $\frac{2L}{1+2N\alpha}$ \cite{Kim+23}
    & \multirow{4}{*}{\makecell[c]{$\frac{2L}{1+2\sum_{k=0}^{N-1}\alpha_k}$ 
    \\ \cite[Rmk. 2]{GSW24}}} \\ 
  \cmidrule(lr){2-4}
    & $\alpha_k\equiv \alpha \in (0,2)$
    & Yes
    & $2L \max \left\{ \frac{1}{1+2N\alpha}, (1-\alpha)^{2N} \right\} $ \cite{RGP24}
    & \\ 
  \cmidrule(lr){2-4}
    & \begin{tabular}[c]{@{}ll@{}}
      Teodor, Fran{\c c}ois and \\
      Panagiotis's schedule
      \end{tabular} \cite[Sec. 2.6.2]{RGP24}
    & Yes
    & $\frac{2L}{1+2\sum_{k=0}^{N-1}\alpha_k}\sim\frac{2L}{1+4N}$ \cite{RGP24}
    & \\ 
  \cmidrule(lr){2-4}
    & Left-side heavy schedule $\mathfrak{h}^{(m)}_{\rm left}$ \cite{GSW24}
    & $N=2^m-1$
    & $\frac{2L}{1+2\sum_{k=0}^{N-1}\alpha_k}\sim \frac{2}{(\rho-1)(1+\rho^{-1/2})} \frac{L}{ N^{\log_2 \rho}}$ \cite{GSW24}
    & \\ 
  \midrule
    \multirow{3}{*}{$\frac{\|\nabla h(x^N)\|}{\|x^0-x^\star\|}$}
    & $\alpha_k\equiv \alpha \in (0,1]$
    & Yes
    & $\frac{L}{1+N\alpha}$ \cite{TV23}
    & \multirow{3}{*}{\makecell[c]{$\frac{L}{1+\sum_{k=0}^{N-1}\alpha_k}$ \\ 
      \cite[Lem. 3.1]{LG24}}}  \\ 
  \cmidrule(lr){2-4}
    & Silver schedule $\pi^{(m)}$ \eqref{def:silver_step} \cite{AP23b}
    & $N=2^m-1$ & $\frac{L}{1+\sum_{k=0}^{N-1}\alpha_k}=\frac{L}{\rho^m}\sim \frac{L}{ N^{\log_2 \rho}}$  [\textbf{Ours, Thm. \ref{thm:silver_gd_tight}}]
    & \\
  \cmidrule(lr){2-4}
    & Mixed schedule $[\mathfrak{h}^{(m)}_{\rm right},\mathfrak{h}^{(m)}_{\rm left}]$ \cite{GSW24}
    & $N=2^{m+1}-2$
    & $\frac{L}{1+\sum_{k=0}^{N - 1}\alpha_k}\sim \frac{\rho}{(\rho-1)(1+\rho^{-1/2})} \frac{L}{ N^{\log_2 \rho}}$ \cite{GSW24}
    & \\ 
  \bottomrule
  \end{tabular}
  \caption{{Convergence bound results of GD method with different stepsize schedules under different measures.
  Here, the objective function $h$ is $L$-smooth and convex, and $\rho=1+\sqrt{2}\approx 2.4142$ denotes the silver ratio.
  Our Theorem \ref{thm:silver_gd_tight} establishes the tight convergence rates of GD method with silver stepsize schedule that have been conjectured in \cite{LG24}.
  The tight rate of silver stepsize schedule under the measure $(h(x^N)-h(x^{\star})) / \left\| x^0-x^{\star} \right\|^{2}$ is very close to that of the right-side heavy schedule $\mathfrak{h}_{\rm right}^{(m)}$ as $(\rho-1)(1+\rho^{-1/2})\approx 2.3244$.
  For the measure $\left\| \nabla h(x^N) \right\| / \left\| x^0-x^{\star} \right\|$, the silver stepsize schedule achieves even better rate than the mixed schedule $[\mathfrak{h}_{\rm right}^{(m)},\mathfrak{h}_{\rm left}^{(m)}]$.
  \label{tab:GD-convergence-compare}}}
\end{sidewaystable}

\subsection{Contributions}
The contributions of this paper are summarized below.   
\begin{enumerate}[leftmargin=*]
    \item We establish tight complexity bounds for RPPA with the constant relaxation schedule $\alpha_k\equiv\alpha\in(0,\sqrt{2}]$ in terms of  function value residual and subgradient norm, see Theorem \ref{thom-tight-analysis} and Remark \ref{rmk:constant_step_subgdnorm}, respectively.  

    \item We establish a tight complexity bound for RPPA with the dynamic relaxation schedule proposed by Teboulle and Vaisbourd in \cite{TV23}. Once again, we use the function value residual as the optimality measure. The main result is presented in Theorem \ref{RPPA_short_dynamic}, which improves the result in Theorem \ref{thom-tight-analysis} by a constant factor of approximately $\sqrt{2}$ as $N\rightarrow \infty$.
 
    \item The original silver stepsize schedule, denoted by $\pi^{(m)}$, was proposed in \cite{AP23b} for GD method. 
    In this paper, we propose two extensions of it, i.e.,  
    the left and the right silver stepsize schedules, denoted respectively by $\underline{\pi}^{(m)}$ and $\overline{\pi}^{(m)}$.
    For  $\pi^{(m)}$, $\underline{\pi}^{(m)}$ and $\overline{\pi}^{(m)}$, we establish tight complexity bounds based on three different optimality measures $\left\| \nabla f^{\lambda}(x^N) \right\| / \left\| x^{0}-x^{\star} \right\|$ (Theorem \ref{thm:RPPA_silver_tight_basic}), $\left\| \nabla f^\lambda (x^N) \right\|^{2} / (f(x^0)-f(x^{\star}))$ (Theorem \ref{thm:subgdnorm-vs-objval}) and  $(f(z^N)-f(x^{\star})) / \left\| x^0-x^{\star} \right\|^{2}$ (Theorem \ref{thm:objval-vs-ptnorm}), respectively.
    The established bounds show that RPPA with all these silver stepsize schedules converges at the same rate  $O(1 / N^{1.2716})$, which is faster than the classical rate $O(1/N)$ for standard GD method with ``short stepsizes"   $\alpha_k\in (0,2)$.
    Furthermore, the established upper bounds are tight since they match the lower bounds derived in Lemma \ref{RPPA_general_lower_bound}. See Section \ref{sec-acc} for details. 

    \item As a byproduct of our research, we establish in Theorem \ref{thm:silver_gd_tight} tight convergence rate results for the GD method with the silver stepsize schedule $\pi^{(m)}$ under two different optimality measures for solving smooth convex optimization problems. These results provide an affirmative answer to the conjecture in \cite[Conjecture 3.2]{LG24}.
\end{enumerate}

\subsection{Notation and organization} 
Our notation is standard. Here, we clarify some of it. 
We denote $\mathbb{R}_{++}$ as the set of all positive real numbers; $\mathbb{R}^d$ represents the standard Euclidean space with dimension $d$; $\mathbb{R}^d_{++}$ is the subset of $\mathbb{R}^d$ consisting of all vectors in $\mathbb{R}^d$ with positive components; $\mathbb{R}^{m\times n}$ stands for the set of $m\times n$ matrices; and $\mathbb{N}_+$ is the set of positive integers.
The standard inner product and the induced norm on $\mathbb{R}^d$ are denoted by  $\left<\cdot, \cdot \right>$ and $\| \cdot \|$ respectively.
The \emph{subgradient} of a convex function $f$ at $x\in\mathbb{R}^d$, denoted by $\partial f(x)$, is given by $\partial f(x) = \{ g \in \mathbb{R}^d \mid f(y)\geq f(x) + \left<g,y-x \right> \text{~~for all~~} y \in \mathbb{R}^d\}$. The \emph{effective domain} of $f$ is given by $\mathrm{dom}(f)=\left\{ x \in \mathbb{R}^{d} \mid f(x)<\infty \right\}$.
For any positive sequences $\{a_k\}$ and $\{b_k\}$, we employ the notation $a_k \sim b_k$ to signify that $\lim_{k\to\infty} a_k / b_k = 1$. 
Other notations will be explained as the paper proceeds.

The remainder of this paper is structured as follows. In Section \ref{sec-pre}, we gather preliminary results that will be useful in the subsequent analysis. Specifically, we derive lower complexity bounds of RPPA with any relaxation schedule under different optimality measures. Section \ref{sec_tight} is dedicated to a thorough analysis of RPPA with a constant relaxation schedule $\alpha_k\equiv\alpha \in(0,\sqrt{2}]$. Section \ref{TV23_schedule} deals with Teboulle and Vaisbourd's relaxation schedule as presented in \cite{TV23}. Section \ref{sec-acc} examines the silver stepsize schedule of Altschuler and Parrilo \cite{AP23b} and two extensions thereof. Finally, concluding remarks are presented in Section \ref{sec-conclusion}.

\section{Preliminaries} \label{sec-pre}
The following identities and inequalities are elementary and easy to verify. They will be employed in our analysis.
For any $x,y,z,w \in \mathbb{R}^d$, $\theta \in \mathbb{R}$ and $\kappa > 0$, there hold
\begin{align}
     2\langle x-y, z-w\rangle & =  \|x-w\|^{2}+\|y-z\|^{2}-\|x-z\|^{2}-\|y-w\|^{2}, \label{id-square} \\
    \|\theta x + (1-\theta)y\|^2 & = \theta \|x\|^2 + (1-\theta) \|y\|^2 - \theta(1-\theta)\|x-y\|^2, \label{id-cvx} \\
    (1+1/\kappa) \|x\|^2+(1+\kappa)\|y\|^2 & \geq \|x+y\|^2   \geq (1-1/\kappa) \|x\|^2+(1-\kappa)\|y\|^2. \label{in-cvx}
\end{align}

Next, we present two lemmas. 
Lemma \ref{in-tv23} is an extension of \cite[Lemma 4]{TV23}, and Lemma \ref{step-tv23} recalls some properties of the stepsize schedule proposed in \cite{TV23}. 
It should be noted that in \cite{TV23}, it is specified that $\gamma \in \{0,1\}$. However, for the conclusion of Lemma \ref{in-tv23} to be valid, it is sufficient to ensure that $\gamma \geq -s$.

\begin{lemma}[Extension of {\cite[Lemma 4]{TV23}}] \label{in-tv23}
For any $x,y,z \in \mathbb{R}^d$, $s>0$, $\gamma \geq -s$ and $v \in \mathbb{R}$.  If $\|x\|^2 \geq \|y\|^2 + s(s+\gamma)\|z\|^2 + sv$ and
 $v \leq 2\langle y,z\rangle -\gamma \|z\|^2$, then there holds  
 $v \leq \|x\|^2/(2s+\gamma)$. 
\end{lemma}

\begin{lemma}[Collected from {\cite{TV23}}] \label{step-tv23}
For $k\geq 1$, define $\alpha_{k}$ recursively as
\begin{equation} \label{dynamic_stepsize}
\alpha_{k} = \frac{1}{2}\Big(-A_{k-1}+\sqrt{A_{k-1}^2+8(A_{k-1}+1)}\Big)
\text{~~with~~} \alpha_{0} = \sqrt{2} \text{~~and~~} 
A_{k-1} := \sum_{i=0}^{k-1} \alpha_{i}.
\end{equation}
Then, it holds for any $k\geq 0$ that 
\begin{align}
\sqrt{2}\leq \alpha_{k} = \frac{2(1+A_{k})}{2+A_{k}} < 2\text{~~and~~} \alpha_{k}^2-2 = (2-\alpha_{k})A_{k-1}. \label{st}
\end{align}
Furthermore, we also have $A_{k-1} \sim 2k$.  \end{lemma}

In fact, the first equality in \eqref{st} is established by \cite[Lemma 13]{TV23}, while the second equality follows directly from the definition of $\alpha_k$. Moreover, it can be verified that $\alpha_k$ is monotonically increasing and approaches $2$ as $k\to\infty$. Thus, $\sqrt{2}\leq \alpha_k< 2$ and $A_{k-1}/(2k)$ converges to $1$. For simplicity, below we refer to the relaxation schedule defined in \eqref{dynamic_stepsize} as TV's schedule.

Next, we establish some lower bounds that are applicable to Algorithm \ref{rppa} with any relaxation schedule $\boldsymbol{\alpha}$. These bounds serve as limitations and set the stage for determining the best possible upper bounds of RPPA when carefully designed relaxation schedules are employed. 
The proofs of these lower bounds are based on constructing examples that match these lower bounds. Particularly, the examples 
have been recognized in \cite{THG17a, TV23} as worst-case instances for PPA.

\begin{lemma}[Lower bounds for RPPA]\label{RPPA_general_lower_bound}
  Let $N \in \mathbb{N}_+$ be any positive integer, $\{(x^k,z^k)\}_{k=0}^N$ be the sequence generated by Algorithm \ref{rppa} with any positive relaxation schedule $\boldsymbol{\alpha} = (\alpha_0,\alpha_1,\ldots,\alpha_{N-1}) \in \mathbb{R}^N_{++}$ for solving \eqref{eq:merely_cvx_problem}. 
  Then, there hold
  \begin{enumerate}
      \item[(i)] $\left\| \nabla f^{\lambda}(x^N)\right\| = \left\| \frac{1}{\lambda} (x^N-z^N) \right\| \geq \frac{1}{1 + \sum_{i=0}^{N-1} \alpha_i} \frac{1}{\lambda} \left\| x^0 - x^\star \right\|$,
      
      \item[(ii)] $f(z^N) - f(x^\star) \geq  \frac{1}{2 \left( 1+\sum_{i=0}^{N-1} \alpha_i\right)} \frac{1}{2\lambda} \left\| x^0 - x^\star \right\|^2$, and
      
      \item[(iii)] $\left\| \nabla f^{\lambda}(x^N)\right\|^2 = \left\| \frac{1}{\lambda} (x^N-z^N) \right\|^2 \geq \frac{1}{1+\sum_{i=0}^{N-1}\alpha_i} \frac{1}{\lambda} \left( f(x^0) - f(x^\star) \right)$.
  \end{enumerate}
\end{lemma}
\begin{proof}
The proof is accomplished by constructing specific functions that match the lower bounds. 
In all three cases, the constructed functions have the form $f(\cdot) = \eta \| \cdot\|$ for some $\eta>0$.
Note that $f$ is convex,  $\mathcal{X}^\star =\{ 0 \}$, $f(x^\star)=0$ and $\partial f (x) =\left\{ \eta x / \|x\| \right\}$ for any $x\neq 0$. Let $x^0$ satisfy $\left\| x^0 \right\|=1$.
  
To prove (i) and (iii), we let $\eta = \big( \lambda ( 1 + \sum_{i=0}^{N-1}\alpha_i)\big)^{-1}$. 
Using \eqref{alg:R-PPA-compact-form}, \eqref{eq:gd_mor_envlp} and $\|x^0\|=1$, by induction we can show that $x^k = x^0 - \sum_{i=0}^{k-1}  \alpha_i \lambda \eta x^i / \|x^i\|=\big( 1 -  \sum_{i=0}^{N-1}  \alpha_i \lambda \eta \big) x^0 \neq 0$ and $z^k = x^k - \lambda \eta x^k / \left\| x^k \right\| = \big( 1- \lambda \eta -  \sum_{i=0}^{N-1}  \alpha_i \lambda \eta\big) x^0 $ for any $0\le k\leq N$.
Therefore, we have shown that $\nabla f^{\lambda}(x^N) =  \big(\lambda(1+\sum_{i=0}^{N-1}  \alpha_i)\big)^{-1}x^0$, which proves (i). 
Moreover, by noting $\|x^0\|=1$, $f(x^\star)=0$ and $f(x^0) = \eta$, it is easy to verify that $\|\nabla f^{\lambda}(x^N)\|^2 
= \big(\lambda ( 1 + \sum_{i=0}^{N-1}\alpha_i)\big)^{-1} (f(x^0)-f(x^\star))$, which proves (iii).
To prove (ii), we set $\eta = \big(2 \lambda (1 + \sum_{i=0}^{N-1}\alpha_i)\big)^{-1}$. 
In this case, we have $z^N =  x^0/2$, and thus   $f(z^N) = \big(4 \lambda (1 + \sum_{i=0}^{N-1} \alpha_i)\big)^{-1}$, which proves (ii) since $\|x^0-x^\star\|=1$.
\end{proof}

If a complexity upper bound established for RPPA with any relaxation schedule coincides with any of the lower bounds provided in Lemma \ref{RPPA_general_lower_bound}, it is considered to be tight.

\section{Tight analysis of RPPA with \texorpdfstring{$\alpha_k\equiv \alpha \in (0,\sqrt{2}]$}{}} \label{sec_tight}

In this section, we establish a tight worst-case complexity bound for RPPA as presented in Algorithm \ref{rppa} with constant relaxation schedule $\alpha_k\equiv \alpha \in (0,\sqrt{2}]$. 
Drawing inspiration from \cite{TV23}, we identify the key monotonically decreasing terms within the analysis that are crucial for our analysis and introduce a double sufficient decrease inequality to simplify the proof.
We now present a useful lemma, which is applicable beyond the case of constant relaxation schedule.

\begin{lemma} \label{lem-basic}
Let $x^{\star} \in \mathcal{X}^{\star}$ be arbitrarily fixed and 
$\{(z^k,x^k)\}_{k\geq 0}$ be the sequence generated by Algorithm \ref{rppa} for solving \eqref{eq:merely_cvx_problem}. Then,  for $k\geq 0$, we have
\begin{align}
&f(x^{\star})  \geq f(z^{k}) + \frac{1}{2\lambda \alpha_k}\|x^{k+1}-x^{\star}\|^2-\frac{1}{2\lambda \alpha_k}\|x^{k}-x^{\star}\|^2 + \frac{2-\alpha_k}{2\lambda}\|z^{k}-x^{k}\|^2, \label{lem-basic-1}\\
&f(z^{k})  \geq f(z^{k+1}) +\frac{1}{2\lambda}\left(\|z^{k+1}-x^{k+1}\|^2+\|z^{k+1}-z^{k}\|^2-(1-\alpha_k)^2\|z^{k}-x^{k}\|^2 \right), \label{lem-basic-2} \\
& \frac{\alpha_{k}-2}{2(1-\alpha_k)}  \|z^{k+1}-z^{k}\|^2 + \frac{\alpha_{k}}{2(1-\alpha_{k})}\|z^{k+1}-x^{k+1}\|^2+\frac{\alpha_{k}(\alpha_{k}-1)}{2}\|z^{k}-x^{k}\|^2 \nonumber\\
& \quad \quad =
\langle (x^{k+1}\!-z^{k+1})-(x^{k}-z^{k}), z^{k+1}-z^{k}\rangle \geq 0. \quad\quad (\alpha_k \neq 1) \label{lem-basic-3}
\end{align}
\end{lemma}

\begin{proof}
For any $k\geq 0$, from $z^{k} = \mathrm{Prox}_{\lambda f}(x^{k})$, we obtain $0 \in \partial f(z^{k}) + \frac{z^{k}-x^{k}}{\lambda}$, and thus
\begin{equation} \label{lem-1-1}
f(x) \geq f(z^{k}) + \frac{1}{\lambda} \langle x^{k}-z^{k}, x-z^{k} \rangle, \quad \forall x\in\mathbb{R}^d.
\end{equation}
Substituting  $x=x^{\star}$ in \eqref{lem-1-1}, we get
\begin{equation} \label{lem-1-2}
\begin{aligned}
f(x^{\star})   \hspace{1.2pt} \geq \hspace{1.2pt} f(z^{k}) + \frac{1}{\lambda} \langle x^{k}-z^{k}, x^{\star}-z^{k} \rangle 
 \stackrel{\eqref{id-square}} = f(z^{k}) + \frac{1}{2\lambda}\left(\|z^{k}-x^{k}\|^2+\|z^{k}-x^{\star}\|^2-\|x^{k}-x^{\star}\|^2 \right).
\end{aligned}
\end{equation}
Since $x^{k+1} = \alpha_k z^{k} + (1-\alpha_k) x^{k}$, using \eqref{id-cvx} and the fact that $x^{k+1}-x^{k}=\alpha_k(z^{k}-x^{k})$, we deduce
\[
\|z^{k}-x^{\star}\|^2 = \frac{1}{\alpha_k}\|x^{k+1}-x^{\star}\|^2 + \frac{\alpha_k-1}{\alpha_k}\|x^{k}-x^{\star}\|^2+ (1-\alpha_k)\|z^{k}-x^{k}\|^2.
\]
Substituting this into \eqref{lem-1-2} directly gives \eqref{lem-basic-1}. Setting $x=z^{k-1}$ in \eqref{lem-1-1} and utilizing $z^{k-1}-x^{k}= (1-\alpha_{k-1})(z^{k-1}-x^{k-1})$, we obtain
\begin{align*}
f(z^{k\!-1}) & \geq f(z^{k}) + \frac{1}{\lambda}\langle x^{k}\!-z^{k}, z^{k\!-1}\!-z^{k}\rangle \stackrel{\eqref{id-square}}= f(z^{k}) + \frac{1}{2\lambda}\left(\|z^{k}-x^{k}\|^2\!+\|z^{k}-z^{k\!-1}\|^2\!-\|x^{k}-z^{k\!-1}\|^2 \right)\\
& = f(z^{k}) + \frac{1}{2\lambda}\left(\|z^{k}-x^{k}\|^2+\|z^{k}-z^{k-1}\|^2-(1-\alpha_{k-1})^2\|z^{k-1}-x^{k-1}\|^2 \right), 
\end{align*}
which verifies \eqref{lem-basic-2}. 
The equality in \eqref{lem-basic-3} follows by
\begin{align*}
\langle (x^{k+1}-z^{k+1})&-(x^{k}-z^{k}), z^{k+1}-z^{k}\rangle = \langle x^{k+1}-z^{k+1}, z^{k+1}-z^{k}\rangle - \langle x^{k}-z^{k}, z^{k+1}-z^{k}  \rangle\\
& \stackrel{(\mathrm{Alg} \ref{alg:relax-step})} =  \langle x^{k+1}-x^{k+1}, z^{k+1}-z^{k}\rangle + \frac{1}{\alpha_{k}-1}\langle x^{k+1}-z^{k},z^{k+1}-z^{k}\rangle \\
& \hspace{6pt} \stackrel{\eqref{id-square}} = \hspace{6pt} \frac{\alpha_{k}-2}{2(1-\alpha_{k})}\|z^{k+1}-z^{k}\|^2 + \frac{\alpha_{k}}{2(1-\alpha_{k})}\|z^{k+1}-x^{k+1}\|^2+\frac{\alpha_{k}}{2(\alpha_{k}-1)}\|z^{k}-x^{k+1}\|^2\\
& \stackrel{(\mathrm{Alg} \ref{alg:relax-step})} = \frac{\alpha_{k}-2}{2(1-\alpha_{k})}\|z^{k+1}-z^{k}\|^2 + \frac{\alpha_{k}}{2(1-\alpha_{k})}\|z^{k+1}-x^{k+1}\|^2+\frac{\alpha_{k}(\alpha_{k}-1)}{2}\|z^{k}-x^{k}\|^2.
\end{align*}
Recall that $\frac{x^{k}-z^{k}}{\lambda}\in \partial f(z^{k})$. Leveraging the monotonicity property of the proximal operator, we can deduce that $\langle (x^{k+1}\!-\!z^{k+1})-\!(x^{k}\!-\!z^{k}), z^{k+1}\!-\!z^{k}\rangle \geq 0$, thereby confirming the assertion of \eqref{lem-basic-3}.
\end{proof}

Next, we prove the monotonicity of $\|z^{k}-x^{k}\|$, which shows that the norm of a specifically selected subgradient element decreases monotonically as $k$ increases.

\begin{lemma}[Monotonicity of $\|z^{k}-x^{k}\|$]
Let $\{(z^k,x^k)\}_{k\geq 0}$ be the sequence generated by Algorithm \ref{rppa} for solving \eqref{eq:merely_cvx_problem} with $\alpha_k\equiv \alpha \in (0,2]$. Then, for any $k\geq 0$, it holds that
\begin{equation} \label{mono-sub}
\|z^{k+1}-x^{k+1}\| \leq \|z^{k}-x^{k}\|.    
\end{equation}
\end{lemma}
\begin{proof}
We consider three cases: (i) $\alpha \in (0,1)$, (ii) $\alpha \in (1,2]$, and (iii) $\alpha = 1$.
\begin{enumerate}
    \item[(i)] For $\alpha \in (0,1)$, applying \eqref{in-cvx} with $x=z^{k+1}-x^{k+1}$, $y=x^{k+1}-z^{k}$ and $\kappa=\frac{1}{1-\alpha}>0$, we obtain 
    \begin{equation}
    \begin{aligned} \label{mono-case1-1}
    \|z^{k+1}-z^{k}\|^2 & \hspace{7.35pt} \geq \hspace{7.35pt} \alpha \|z^{k+1}-x^{k+1}\|^2 - \frac{\alpha}{1-\alpha}\|x^{k+1}-z^{k}\|^2  \\
    & \stackrel{(\mathrm{Alg} \ref{alg:relax-step})} = \alpha\|z^{k+1}-x^{k+1}\|^2-\alpha(1-\alpha)\|z^{k}-x^{k}\|^2.
    \end{aligned}
    \end{equation}
    \item[(ii)] For $\alpha \in (1,2]$, similarly, applying \eqref{in-cvx} with $x=z^{k+1}\!-x^{k+1}$, $y=x^{k+1}\!-z^{k}$ and $\kappa\!=\frac{1}{\alpha-1}>0$ gives
    \begin{equation} 
    \begin{aligned}  \label{mono-case1-2}
    \|z^{k+1}-z^{k}\|^2 & \hspace{7.35pt} \leq \hspace{7.35pt}\alpha \|z^{k+1}-x^{k+1}\|^2 + \frac{\alpha}{\alpha-1}\|x^{k+1}-z^{k}\|^2  \\
    & \stackrel{(\mathrm{Alg} \ref{alg:relax-step})} = \alpha\|z^{k+1}-x^{k+1}\|^2 + \alpha(\alpha-1)\|z^{k}-x^{k}\|^2.
    \end{aligned}
    \end{equation}
\end{enumerate}
For both cases (i) and (ii), we derive
\begin{align*}
    & \quad  \alpha \|z^{k}-x^{k}\|^2- \alpha \|z^{k+1}-x^{k+1}\|^2 \\
    & \hspace{16.8pt} = \hspace{16.8pt} \frac{\alpha-2}{1\!-\!\alpha}\left(\alpha \|z^{k+1}\!-x^{k+1}\|^2\!-\!\alpha(1\!-\!\alpha)\|z^{k}-x^{k}\|^2\right) + \frac{\alpha}{1\!-\!\alpha}\|z^{k+1}\!-x^{k+1}\|^2+ \alpha(\alpha\!-\!1) \|z^{k}-x^{k}\|^2 \\
    & \stackrel{\eqref{mono-case1-1} \text{ or }\eqref{mono-case1-2}}\geq \frac{\alpha-2}{1\!-\!\alpha}\|x^{k+1}-x^{k}\|^2 + \frac{\alpha}{1\!-\!\alpha}\|z^{k+1}-x^{k+1}\|^2+ \alpha(\alpha\!-\!1) \|z^{k}-x^{k}\|^2 \stackrel{\eqref{lem-basic-3}}\geq 0,
\end{align*}
where the ``$=$" can be verified straightforwardly. This
confirms \eqref{mono-sub} for $\alpha \in (0,1) \cup (1,2]$. 
\begin{enumerate}
    \item[(iii)] For $\alpha = 1$, we have $x^{k+1}=z^{k}$. Using the monotonicity of the proximal operator, we derive
    \begin{align*}
    0 & \leq \langle (z^{k}-z^{k+1})-(z^{k-1}-z^{k}), z^{k+1}-z^{k}\rangle = \langle (2z^{k}-z^{k-1})-z^{k+1}, z^{k+1}-z^{k}\rangle \\
    & \leq \|z^{k}-z^{k-1}\|^2 - \|z^{k+1}-z^{k}\|^2
    = \|z^{k}-x^{k}\|^2 - \|z^{k+1}-x^{k+1}\|^2.
    \end{align*}
\end{enumerate}
In summary, we have shown that \eqref{mono-sub} holds for all $k\geq 0$. 
\end{proof}

\begin{lemma}[Double sufficient decrease] \label{thom-2SD}
Let $\{(z^k,x^k)\}_{k\geq 0}$ be the sequence generated by Algorithm \ref{rppa} for solving \eqref{eq:merely_cvx_problem} with  $\alpha_k\equiv \alpha \in (0,2)$. Then, for any $k\geq 0$, there hold
\begin{enumerate}
    \item  $f(z^{k}) \geq f(z^{k+1}) + \frac{\alpha+1}{2\lambda}\|z^{k+1}-x^{k+1}\|^2 + \frac{\alpha-1}{2\lambda}\|z^{k}-x^{k}\|^2$ if $\alpha \in (0,1]$, and \dotfill $(\mathrm{2SD\!-\!1})$
    \item $f(z^{k}) \geq f(z^{k+1}) + \!\frac{1}{2\lambda}\frac{2}{2-\alpha}\|z^{k+1}\!-\!x^{k+1}\|^2 \!+\! \frac{1}{2\lambda}\frac{2(\alpha-1)^2}{\alpha-2}\|z^{k}-x^{k}\|^2$ if $\alpha \in (1,2)$. \dotfill $(\mathrm{2SD\!-\!2})$
\end{enumerate}
\end{lemma} 
\begin{proof}
We first address the case $\alpha \in (0,1]$, the proof of which is divided into two subcases: $\alpha \in (0,1)$ and $\alpha=1$. 
For $\alpha \in (0,1)$, combining \eqref{lem-basic-2} with \eqref{mono-case1-1} directly yields (2SD-1). 
For $\alpha = 1$, we have $x^{k+1}=z^{k}$. Thus,  \eqref{lem-basic-2} reduces to 
$ f(z^{k}) \geq f(z^{k+1}) +\frac{1}{\lambda} \|z^{k+1}-x^{k+1}\|^2$, which is exactly (2SD-1). 

Now, consider the case $\alpha \in (1,2)$. It follows from \eqref{lem-basic-3} that
\[
\|z^{k+1}-z^{k}\|^2\geq \frac{\alpha}{2-\alpha}\|z^{k+1}-x^{k+1}\|^2 + \frac{\alpha(\alpha-1)^2}{\alpha-2}\|z^{k}-x^{k}\|^2.
\]
Combining the above inequality with \eqref{lem-basic-2} directly demonstrates (2SD-2). 
\end{proof}

Now, we are prepared to establish the tight complexity bound for RPPA 
with $\alpha_k\equiv\alpha\in(0,\sqrt{2}]$ by using the function value residual as the optimality measure. 

\begin{thom}[Tight complexity bound for RPPA with {$\alpha_k\equiv\alpha\in(0,\sqrt{2}]$}] \label{thom-tight-analysis}
Let $x^{\star} \in \mathcal{X}^{\star}$  and $N\in \mathbb{N}_+$ be arbitrarily fixed. 
Let $z^{N}$ be the output after $N$ iterations of Algorithm \ref{rppa} for solving \eqref{eq:merely_cvx_problem} with constant relaxation schedule $\alpha_k\equiv \alpha \in (0,\sqrt{2}]$. Then, it holds that 
\begin{equation}\label{thm3.1-key}
f(z^{N}) - f(x^{\star}) \leq \frac{1}{\alpha N+1}\frac{1}{4\lambda}\|x^0-x^{\star}\|^2.  
\end{equation}
Furthermore, this bound is tight.
\end{thom}
\begin{proof}
Reorganizing the terms in \eqref{lem-basic-1} gives  
\[
\frac{1}{2\lambda\alpha} \|x^{k}-x^{\star}\|^2 - \frac{1}{2\lambda\alpha} \|x^{k+1}-x^{\star}\|^2 \geq f(z^{k})-f(x^{\star})+\frac{2-\alpha}{2\lambda}\|z^{k}-x^{k}\|^2, \; \forall k\geq 0.
\]
Summing up the above inequality over $k=0,\ldots,N$, we obtain
\begin{align} \label{thom-1-1}
\frac{1}{2\lambda\alpha}\|x^{0}-x^{\star}\|^2 &- \frac{1}{2\lambda\alpha}\|x^{N+1}-x^{\star}\|^2  \geq \sum_{i=0}^{N}\left(f(z^{i})-f(x^{\star})+\frac{2-\alpha}{2\lambda}\|z^{i}-x^{i}\|^2\right) \nonumber \\
& = \sum_{i=0}^{N}\left(f(z^{i})-f(z^{N})+\frac{2-\alpha}{2\lambda}\|z^{i}-x^{i}\|^2\right) + (N+1)(f(z^{N})-f(x^{\star})).    
\end{align}
Next, we separate the proof into two cases: $\alpha \in (0,1]$ and $\alpha \in (1,\sqrt{2}]$. For $\alpha \in (0,1]$, utilizing (2SD-1) in Lemma \ref{thom-2SD}, we obtain for all $0\leq i< N$ that
\begin{align} \label{thom-1-2}
f(z^{i})-f(z^{N}) + \frac{2-\alpha}{2\lambda}\|z^{i}-x^{i}\|^2 & \hspace{3.2pt} \geq \hspace{3.2pt} \frac{\alpha+1}{2\lambda}\|z^{N}-x^{N}\|^2 + \frac{\alpha}{\lambda} \sum_{k=i+1}^{N-1} \|z^{k}-x^{k}\|^2 + \frac{1}{2\lambda}\|z^{i}-x^{i}\|^2 \nonumber\\
& \stackrel{\eqref{mono-sub}}\geq \left(\frac{(N-i)\alpha}{\lambda}+\frac{2-\alpha}{2\lambda}\right)\|z^{N}-x^{N}\|^2.
\end{align}
For  $\alpha \in (1,\sqrt{2}]$, we can also establish \eqref{thom-1-2}. In fact, in this case, by utilizing (2SD-2) in Lemma \ref{thom-2SD}, it yields for all $0\leq i< N$ that
\begin{align*}
f(z^{i})-f(z^{N}) + \frac{2-\alpha}{2\lambda}\|z^{i}-x^{i}\|^2 & \hspace{3.2pt} \geq \hspace{3.2pt} \frac{1}{2\lambda}\frac{2}{2-\alpha}\|z^{N}-x^{N}\|^2 \!+\! \frac{\alpha}{\lambda} \sum_{k\!=i+\!1}^{N-1} \|z^{k}-x^{k}\|^2 \!+\! \frac{1}{2\lambda}\frac{\alpha^2-2}{\alpha-2}\|z^{i}-x^{i}\|^2 \nonumber\\
& \stackrel{\eqref{mono-sub}}\geq \left(\frac{(N-i)\alpha}{\lambda}+\frac{2-\alpha}{2\lambda}\right)\|z^{N}-x^{N}\|^2,
\end{align*}
where the second ``$\geq$" also uses $\alpha \leq \sqrt{2}$.  
Substituting \eqref{thom-1-2} into \eqref{thom-1-1}, we obtain
\begin{align} \label{thom-1-3}
\frac{1}{2\lambda\alpha}\|x^{0}-x^{\star}\|^2 &  - \frac{1}{2\lambda\alpha}\|x^{N+1}-x^{\star}\|^2 \nonumber\\
& \geq (N+1)\left[\frac{\alpha (N+1) + 2(1-\alpha)}{2\lambda}\|z^{N}-x^{N}\|^2+(f(z^{N})-f(x^{\star}))\right].
\end{align}
Substituting $k=N$ into \eqref{lem-1-2}, we drive
\begin{align}
f(x^{\star}) & \hspace{7.35pt} \geq \hspace{7.35pt} f(z^{N})+\frac{1}{\lambda}\langle x^{N}-z^{N},x^{\star}-z^{N}\rangle = f(z^{N})+\frac{1}{\lambda}\langle x^{N}-z^{N},x^{\star}-x^{N+1}+x^{N+1}-z^{N}\rangle \nonumber  \\
& \stackrel{(\mathrm{Alg} \ref{alg:relax-step})}= f(z^{N})+\frac{1}{\lambda}\langle x^{N}-z^{N},x^{\star}-x^{N+1}\rangle + \frac{1-\alpha}{\lambda}\|z^{N}-x^{N}\|^2. \label{thom-1-4}
\end{align}
Combining \eqref{thom-1-3} and \eqref{thom-1-4}, setting $v = 2\lambda (f(z^{N})-f(x^{\star}))$, $x = x^{0}-x^{\star}$, $y = x^{N+1}-x^{\star}$, $z = z^{N}-x^{N}$, $s=\alpha(N+1)$ and $\gamma = 2(1-\alpha)$ in Lemma \ref{in-tv23}, we directly obtain the desired result \eqref{thm3.1-key} by noting that $s+\gamma >0$ holds for any $\alpha \in (0,\sqrt{2}]$ and $N\geq 1$.
Finally, the tightness of \eqref{thm3.1-key} can be easily verified by noticing that it matches the lower bound given in item (ii) of Lemma \ref{RPPA_general_lower_bound}. 
\end{proof}

\begin{remark}[Complexity bound for subgradient norm $\left\| \nabla f^{\lambda}(x^N) \right\|$] \label{rmk:constant_step_subgdnorm}
Let $\alpha_k\equiv \alpha \in (0,\sqrt{2}]$. 
Since, by \eqref{eq:gd_mor_envlp},  we have 
$\nabla f^{\lambda}(x^N)\in \partial f(z^N)$, $\left\| \nabla f^{\lambda}(x^N) \right\|$ can be viewed as the subgradient norm residual of $f$ at $z^N$. By slightly modifying the previous analysis, we can derive a tight complexity bound for $\left\| \nabla f^{\lambda}(x^N) \right\|$.
First, by summing the inequality \eqref{lem-basic-1} over $k=0,1, \ldots, N-1$, we can obtain a result analogous to \eqref{thom-1-1}. Then, by following the same reasoning as in Theorem \ref{thom-tight-analysis}, we can derive an inequality  similar to \eqref{thom-1-3}:
\begin{align} \label{remark-2-1}
    \frac{1}{2\lambda\alpha}\|x^{0}-x^{\star}\|^2 &  - \frac{1}{2\lambda\alpha}\|x^{N}-x^{\star}\|^2 \geq N\left[\frac{\alpha N + 2}{2\lambda}\|z^{N}-x^{N}\|^2+(f(z^{N})-f(x^{\star}))\right]. 
\end{align}
Next, by setting $k=N$ in \eqref{lem-1-2}, leveraging the fact that $f(z^N)-f(x^{\star})\ge 0$ and rearranging terms, we obtain
\begin{align}  \label{remark-2-2}
  \left\| x^N-x^{\star} \right\|^{2} \ge \left\| z^N-x^N \right\|^{2} + \left\| z^N-x^{\star} \right\|^{2}.
\end{align} 
By combining \eqref{remark-2-2} with \eqref{remark-2-1}, followed by eliminating certain non-negative terms from the right-hand side of \eqref{remark-2-1}, we achieve the bound $\|\nabla f^{\lambda}(x^N)\|^2 \leq \frac{1}{\lambda^2(\alpha N+1)^2}\|x^{0}-x^{\star}\|^2$ for any $\alpha \in (0,\sqrt{2}]$.  The tightness of this bound can be verified by item (i) of Lemma \ref{RPPA_general_lower_bound} as well. 
\end{remark}

\section{Tight analysis of RPPA with TV's schedule \texorpdfstring{\eqref{dynamic_stepsize}}{}} \label{TV23_schedule}

In this section,  we establish a tight worst-case complexity bound for RPPA with TV's relaxation schedule as presented in \eqref{dynamic_stepsize}.  
As in Theorem \ref{thom-tight-analysis}, function value residual is used as the optimality measure, and the established bound is approximately $\sqrt{2}$ times better as $N\rightarrow\infty$.

\begin{thom} \label{RPPA_short_dynamic}
Let $x^{\star} \in \mathcal{X}^{\star}$  and $N\in \mathbb{N}_+$ be arbitrarily fixed. 
Let $z^{N}$ be the output after $N$ iterations of Algorithm \ref{rppa} for solving \eqref{eq:merely_cvx_problem} with TV's relaxation schedule defined in \eqref{dynamic_stepsize}. Then, it holds that 
\begin{equation}\label{thm4.1-key}
f(z^{N}) - f(x^{\star}) \leq \frac{1}{A_{N-1}+1}\frac{1}{4\lambda}\|x^0-x^{\star}\|^2, 
\end{equation}  
where $A_{N-1}$ is defined in \eqref{dynamic_stepsize}. 
Furthermore, this bound is tight.
\end{thom}
\begin{proof}
Similar to the proof of (2SD-2) in Lemma \ref{thom-2SD}, by utilizing \eqref{lem-basic-2} and \eqref{lem-basic-3}, we can show
\[
f(z^{k}) \geq f(z^{k+1}) + \frac{1}{2\lambda}\frac{2}{2-\alpha_{k}}\|z^{k+1}-x^{k+1}\|^2 + \frac{1}{2\lambda}\frac{2(\alpha_{k}-1)^2}{\alpha_{k}-2}\|z^{k}-x^{k}\|^2, \quad \forall k\geq 0, 
\]
which further implies for all $0\leq i< N$ that
\begin{equation}
\begin{aligned} \label{thom-2-1}
f(z^{i})-f(z^{N}) + \frac{2-\alpha_{i}}{2\lambda}\|z^{i}-x^{i}\|^2  \geq & \frac{1}{2\lambda}\frac{2}{2-\alpha_{N-1}}\|z^{N}-x^{N}\|^2 + \frac{1}{2\lambda}\frac{\alpha_{i}^2-2}{\alpha_{i}-2}\|z^{i}-x^{i}\|^2 \\
&+ \frac{1}{2\lambda}\sum_{k\!=i+\!1}^{N-1} \left(\tfrac{2}{2-\alpha_{k-1}}+\tfrac{2(\alpha_{k}-1)^2}{\alpha_{k}-2}\right) \|z^{k}-x^{k}\|^2.    
\end{aligned}    
\end{equation}
Rearranging \eqref{lem-basic-1} gives $\frac{1}{2\lambda} \|x^{k}-x^{\star}\|^2 - \frac{1}{2\lambda} \|x^{k+1}-x^{\star}\|^2 \geq \alpha_{k} (f(z^{k})-f(x^{\star})+\frac{2-\alpha_{k}}{2\lambda}\|z^{k}-x^{k}\|^2)$ for all $k\geq 0$. Summing up this inequality over $k=0,1,\ldots,N$, we obtain
\begin{align*}
\frac{1}{2\lambda}\|x^{0}\!-\!x^{\star}\|^2 \!-\! \frac{1}{2\lambda}\|x^{N+1}\!-\!x^{\star}\|^2  \geq \sum_{i=0}^{N}\alpha_{i}\Big(f(z^{i})\!-\!f(z^{N})+\frac{2-\alpha_i}{2\lambda}\|z^{i}\!-\!x^{i}\|^2\Big) +  A_{N} (f(z^{N})\!-\!f(x^{\star})).   
\end{align*}
Making use of \eqref{thom-2-1}, the above inequality further implies
\begin{align}
&\frac{1}{2\lambda}\|x^{0}\!-\!x^{\star}\|^2 \!-\! \frac{1}{2\lambda}\|x^{N+1}\!-\!x^{\star}\|^2  \geq \tfrac{\alpha_{N}(2-\alpha_{N})}{2\lambda}\|z^{N}-x^{N}\|^2 + A_{N} (f(z^{N})-f(x^{\star})) \nonumber \\
& +\frac{1}{2\lambda}\sum_{i=0}^{N-1}\alpha_{i}\Big[ \tfrac{2}{2-\alpha_{N-1}}\|z^{N}-x^{N}\|^2 + \tfrac{\alpha_{i}^2-2}{\alpha_{i}-2}\|z^{i}-x^{i}\|^2 + \sum_{k\!=i+\!1}^{N-1} \left(\tfrac{2}{2-\alpha_{k-1}}+\tfrac{2(\alpha_{k}-1)^2}{\alpha_{k}-2}\right) \|z^{k}-x^{k}\|^2 \Big] \nonumber\\
= & \frac{1}{2\lambda}\!\Big[\!\left(\tfrac{2A_{N-1}}{2-\alpha_{N-1}}\!+\!\alpha_{N}(2\!-\!\alpha_{N})\!\right)\!\|z^{N}\!-\!x^{N}\|^2\!+\!\sum_{k=1}^{N\!-\!1} \!\left(\tfrac{2A_{k-1}}{2-\alpha_{k-1}}\!+\!\tfrac{2A_{k-1}(\alpha_{k}-1)^2}{\alpha_{k}-2}\!+\!\tfrac{\alpha_{k}(\alpha_{k}^2-2)}{\alpha_{k}-2}\!\right)\!\|z^{k}\!-\!x^{k}\|^2\! \nonumber\\
& \qquad  +\! \tfrac{\alpha_{0}(\alpha_{0}^2-2)}{\alpha_{0}-2}\|z^{0}\!-\!x^{0}\|^2\!\Big] 
 +  A_{N} (f(z^{N})-f(x^{\star})). \label{thom-2-2}
\end{align}
Applying the results from \eqref{st}, we can verify for $k=1,\ldots,N-1$ that
\begin{align*}
\tfrac{2A_{k-1}}{2-\alpha_{k-1}}&\!+\!\tfrac{2A_{k-1}(\alpha_{k}-1)^2}{\alpha_{k}-2}\!+\!\tfrac{\alpha_{k}(\alpha_{k}^2-2)}{\alpha_{k}-2}
= A_{k-1}(2+A_{k-1})+\tfrac{1}{\alpha_{k}-2}[2A_{k-1}(\alpha_{k}-1)^2+\alpha_{k}(2A_{k-1}-\alpha_{k}A_{k-1})]\\
& = A_{k-1}(2+A_{k-1}) + \tfrac{1}{\alpha_{k}-2} [A_{k-1}(\alpha_{k}^2-2\alpha_{k})+2A_{k-1}] =A_{k-1}(2+A_{k-1})+A_{k-1}(\alpha_{k}+\tfrac{2}{\alpha_{k}-2})\\
& = A_{k-1}(2+A_{k-1})+A_{k-1}(\alpha_{k}-2-A_{k})=0,
\end{align*}
as well as 
\begin{align*}
\tfrac{2A_{N-1}}{2-\alpha_{N-1}}&\!+\!\alpha_{N}(2\!-\!\alpha_{N})\! =\!  A_{N-1}(2+A_{N-1})\!+\!2\alpha_{N}-\alpha_{N}^2 \!=\! A_{N-1}^2\!-\!\alpha_{N}^2 \!+\! 2A_{N}= A_{N}(A_{N}+2(1\!-\!\alpha_{N})).
\end{align*}
Hence, \eqref{thom-2-2} can be simplified to 
\begin{align}
\frac{1}{2\lambda}\|x^{0}\!-\!x^{\star}\|^2 \!-\! \frac{1}{2\lambda}\|x^{N+1}\!-\!x^{\star}\|^2 \geq \frac{A_{N}(A_{N}+2(1\!-\!\alpha_{N})) }{2\lambda}\|z^{N}-x^{N}\|^2 \!+\! A_{N} (f(z^{N})\!-\!f(x^{\star})). \label{thom-2-3}
\end{align}
Similar to \eqref{thom-1-4}, we can derive
\begin{align} \label{thom-2-4}
f(x^{\star}) & \geq f(z^{N})+\frac{1}{\lambda}\langle x^{N}-z^{N},x^{\star}-x^{N+1}\rangle + \frac{1-\alpha_{N}}{\lambda}\|z^{N}-x^{N}\|^2.
\end{align}
By combining \eqref{thom-2-3} and \eqref{thom-2-4}, and setting $v = 2\lambda (f(z^{N})-f(x^{\star}))$, $x = x^{0}-x^{\star}$, $y = x^{N+1}-x^{\star}$, $z = z^{N}-x^{N}$, $s=A_{N}$ and $\gamma = 2(1-\alpha_{N})$ in Lemma \ref{in-tv23},  the desired upper bound result \eqref{thm4.1-key} can be derived straightforwardly.
Finally, the tightness of \eqref{thm4.1-key} follows from item (ii) of Lemma \ref{RPPA_general_lower_bound} and the definition of $A_{N-1}$ in Lemma \ref{in-tv23}.
\end{proof}

\section{Accelerate RPPA via silver stepsize schedules} \label{sec-acc}
In this section, we analyze the performance of RPPA using the silver stepsize schedule from \cite{AP23b}. Additionally, we propose two modified schedules based on the silver stepsize schedule.
These schedules permit the stepsizes to be considerably larger than $2$ and enable higher-order convergence rates of $O( 1/ N^{1.2716})$ for either the function value residual or the subgradient norm, surpassing the standard $O(1 / N)$ rate.

Before deriving the results for RPPA, we first establish the tight complexity bound for GD method with the silver stepsize schedule for minimizing smooth convex functions.
Part of these results can be directly utilized to establish the tight convergence rate of RPPA with the silver stepsize schedule.

\subsection{Tight convergence rate of GD method with the silver stepsize schedule}\label{subsec:tight_silver_gd}

Consider minimizing an $L$-smooth convex function $h$ via the $N$-step GD method \eqref{alg:N-step-GD} with stepsize schedule $\boldsymbol{\alpha} = (\alpha_0, \alpha_1, \ldots, \alpha_{N-1})\in\mathbb{R}^N_{++}$ normalized by $L$.
For any $m\ge 1$, the silver stepsize schedule, denoted by $\pi^{(m)}$, was first proposed for the GD method in \cite{AP23b}.
The $m$-th silver stepsize schedule $\pi^{(m)}$ is a vector in $\mathbb{R}^{2^m-1}_{++}$, and it is recursively defined as follows:
\begin{equation} \label{def:silver_step}
  \pi^{(1)} := [\sqrt{2}] \text{~~and~~} \pi^{(m+1)}:= [\pi^{(m)}, 1+\rho^{m-1}, \pi^{(m)}], \quad \forall m\ge 1,
\end{equation}
where $\rho=1+\sqrt{2}$ is the silver ratio satisfying $\rho^{2}=2\rho+1$. 
For convenience, the components of $\pi^{(m)}$ are indexed by $i\in \{0,1,\ldots,2^m-2\}$. 
By this definition, we have $\sum_{i=0}^{2^m-2} \pi_i^{(m)} = \rho^{m}-1$. 
In \cite{AP23b}, the following convergence rate for GD method with silver stepsize schedule was established.
\begin{thom}[{\cite[Theorem 1.1]{AP23b}}]
  For any fixed $m\ge 1$, let $N=2^m-1$. Then, for any $L$-smooth convex function $h:\mathbb{R}^{d}\to \mathbb{R}$ and any $x^{0} \in \mathbb{R}^{d}$, there holds
  $$
  h(x^{N}) - h(x^\star) \le  \frac{L\left\| x^0-x^\star \right\|^{2}}{1+\sqrt{4\rho^{2m-3}-3}} \le \frac{L\left\| x^0-x^\star \right\|^{2}}{2\rho^{\log_2 N}} \approx \frac{L\left\| x^0-x^\star \right\|^{2}}{2 N^{1.2716}},
  $$
  where $x^{\star}$ denotes an arbitrarily fixed minimizer of $h$, and $x^N$ denotes the output of the $N$-step GD method \eqref{alg:N-step-GD} using the silver stepsize schedule $\boldsymbol{\alpha} = \pi^{(m)}$.
\end{thom}

However, as discussed in \cite[Section 3.1]{LG24}, the above convergence rate is not tight. By \cite[Conjecture 3.2]{LG24}, GD method with the silver stepsize schedule $\pi^{(m)}$ was conjectured to have the following tight rates: 
\[
h(x^N) - h(x^\star) \le \frac{L \left\| x^0-x^\star \right\|^{2}} 
{ 4 \rho^{m}-2 }
\text{~~and~~}
\left\| \nabla h(x^N)\right\| \le \frac{ L \left\| x^0-x^\star \right\|}{\rho^{m}}.
\]
In the following, we will prove this conjecture is true. To do this, we first introduce the following definition of $Q^h_{x,y}$, whose nonnegativity plays a fundamental role in the proof.

\begin{definition}
  Given any $L$-smooth convex function $h$ and $x,y \in \mathbb{R}^{d}$, we define
\begin{equation}\label{def:Qh_xy}
  Q^h_{x,y} := h(x) - h(y) - \left<\nabla h(y), x- y \right> - \frac{1}{2L} \left\| \nabla h(x) - \nabla h(y) \right\|^{2},
\end{equation}
which is nonnegative due to $L$-smoothness and convexity of $h$ \cite[Theorem 2.1.5]{Nes18}.
\end{definition}

Next, we recall \cite[Lemma 1]{GSW24}, which is useful in our analysis. 
\begin{lemma}[{\cite[Lemma 1]{GSW24}}] \label{lem:silver_basic_lem}
  For any fixed $m\ge 1$, let $N=2^m-1$. Furthermore, define a matrix $A^{(m)} \in \mathbb{R}^{2^m\times 2^m}$, whose components are indexed by
   $(i,j)\in \{0,1, \ldots, 2^m-1\}^2$, 
  recursively by
  \begin{equation}\label{def:A_mat_recur}
      A^{(1)}: = \begin{bmatrix} 0 & \rho \\ 1 & 0 \\\end{bmatrix},   \text{~~and~~} 
      A^{(i+1)}:=\begin{bmatrix} A^{(i)} &  \\  & \rho^{2} A^{(i)} \\\end{bmatrix} +
      \begin{bmatrix} \mathbf{0}_{(2^{i}-1)\times (2^{i}-1)} &  &  &  \\  & 0  & \rho \pi^{(i)} & \rho \\  &  & \mathbf{0}_{(2^{i}-1)\times (2^{i}-1)} &  \\  & \rho^i & \rho \pi^{(i)} & 0 \\\end{bmatrix},
  \end{equation}
  for $i=1,2,\ldots,m-1$. 
  Let $h:\mathbb{R}^{d}\to \mathbb{R}$ be any $L$-smooth convex function, $x^\star$ be any fixed minimizer of $h$ and $x^0\in\mathbb{R}^d$ be any initial point. 
  Let $\{ x^k \}_{k=0}^{N}$ be the sequence generated by the $N$-step GD method \eqref{alg:N-step-GD} with stepsize schedule $\boldsymbol{\alpha}=\pi^{(m)}$. Then, we have
    \begin{align}
      0\le \sum_{i,j \in \{ 0,1, \ldots ,2^m-1 \}}  A^{(m)}_{i,j} Q^h_{x^i,x^j} = & -\sum_{i=0}^{2^m-2}\alpha_i \Big( h(x^{2^m-1}) - h(x^i) - \frac{1}{2L} \left\| \nabla h(x^i) \right\|^{2}-\left<\nabla h(x^i), x^0-x^i \right>\Big) \nonumber \\
      &- \frac{L}{2} \left\| x^{2^m-1} - x^0 \right\|^{2} - \frac{\rho^m(\rho^m-1)}{2L} \left\| \nabla h(x^{2^m-1}) \right\|^{2}. \label{eq:silver-basic-for-gd}
    \end{align} 
\end{lemma}

Using Lemma \ref{lem:silver_basic_lem}, we can then prove the validity of Conjecture 3.2 in  \cite{LG24}.
\begin{thom}[{\cite[Conjecture 3.2]{LG24}}] \label{thm:silver_gd_tight}
  For any fixed $m\ge 1$, let $N=2^m-1$.  Let $h:\mathbb{R}^{d}\to \mathbb{R}$ be any $L$-smooth convex function, $x^\star$ be any fixed minimizer of $h$, $x^0\in\mathbb{R}^d$ be any initial point. Let $\{ x^k \}_{k=0}^{N}$ be the sequence generated by the $N$-step GD method \eqref{alg:N-step-GD} with stepsize schedule $\boldsymbol{\alpha}=\pi^{(m)}$. Then, we have 
  \begin{align}
    h(x^N) - h(x^\star) &\le \frac{L \left\| x^0-x^\star \right\|^{2}}{2 \Big( 2\sum_{i=0}^{2^m-2} \pi^{(m)}_i + 1 \Big) } =\frac{L \left\| x^0-x^\star \right\|^{2}}{4 \rho^{m}-2}, \label{silver_tight_objval}\\
    \left\| \nabla h(x^N)\right\| &\le \frac{L \left\| x^0-x^\star \right\|}{ \sum_{i=0}^{2^m-2} \pi^{(m)}_i + 1} =\frac{L \left\| x^0-x^\star \right\|}{\rho^{m}}. \label{silver_tight_gdnorm}
  \end{align}
  Furthermore, these bounds are tight.
\end{thom}
\begin{proof}
  To prove \eqref{silver_tight_objval}, we note that $0\le Q_{x^{\star}, x^i}^h=h(x^\star)-h(x^i) - \left<\nabla h(x^i), x^\star - x^i \right> - \frac{1}{2L} \left\| \nabla h(x^i) \right\|^{2}$. Thus, Lemma \ref{lem:silver_basic_lem} implies
  \begin{align*}
    0 \hspace{3.2pt} \le \hspace{3.2pt} & \sum_{i,j \in \{ 0,1, \ldots ,2^m-1 \}}  A^{(m)}_{i,j} Q^h_{x^i,x^j} + \sum_{i=0}^{2^m-2} Q^{h}_{x^\star,x^i} + \rho^m Q^{h}_{x^\star, x^{2^m-1}} \\
    \overset{\eqref{eq:silver-basic-for-gd}}{=} & \sum_{i=0}^{2^m-2} \pi^{(m)}_i \left( h(x^\star)-h(x^{2^m-1}) - \left<\nabla h(x^i), x^\star - x^0 \right> \right)- \frac{L}{2} \left\| x^{2^m-1} - x^0 \right\|^{2} - \frac{\rho^{2m}}{2L} \left\| \nabla h(x^{2^m-1}) \right\|^{2} \\
    & + \rho^m \Big( h(x^\star)-h(x^{2^m-1}) - \big<\nabla h(x^{2^m-1}), x^\star - x^{2^m-1}\big> \Big)\\
    \hspace{2pt} \overset{(*)}{=} \hspace{2pt} & (2\rho^m-1) \Big( h(x^\star)-h(x^{2^m-1}) \Big) + L \Big< x^{2^m-1}-x^0, x^\star - x^0\Big> + \frac{L}{2} \left\| x^0-x^\star \right\|^{2} - \frac{L}{2} \left\| x^0-x^\star \right\|^{2}\\
    & - \frac{L}{2} \left\| x^{2^m-1} - x^0 \right\|^{2} - \frac{\rho^{2m}}{2L} \left\| \nabla h(x^{2^m-1}) \right\|^{2} - \rho^{m} \left<\nabla h(x^{2^m-1}), x^\star - x^{2^m-1} \right>\\
    \hspace{3.2pt} = \hspace{3.2pt} & (2\rho^m-1) \left( h(x^\star)-h(x^{2^m-1}) \right) + \frac{L}{2} \left\| x^0-x^\star \right\|^{2} - \frac{L}{2} \left\| x^{2^m-1} - \frac{\rho^m}{L} \nabla h(x^{2^m-1}) -x^\star \right\|^{2},
  \end{align*}
  where we have used $\sum_{i=0}^{2^m-2} \pi^{(m)}_i = \rho^m-1$ and 
  $-\sum_{i=0}^{2^m-2} \pi^{(m)}_i \nabla h(x^i) = L(x^{2^m-1}-x^0)$ in $(*)$. Reorganizing the terms yields the desired upper bound in \eqref{silver_tight_objval}.

  To prove \eqref{silver_tight_gdnorm}, we note that $0\le Q_{x^{\star}, x^i}^h=h(x^\star)-h(x^i) - \left<\nabla h(x^i), x^\star - x^i \right> - \frac{1}{2L} \left\| \nabla h(x^i) \right\|^{2}$ and
  $0\le Q_{ x^{2^m-1}, x^{\star}}^h=h(x^{2^m-1}) -h(x^\star) - \frac{1}{2L} \left\| \nabla h(x^{2^m-1}) \right\|^{2}$.
  Thus, Lemma \ref{lem:silver_basic_lem} implies
  \begin{align*}
    0 \hspace{3.2pt} \le \hspace{3.2pt} & \sum_{i,j \in \{ 0,1, \ldots ,2^m-1 \}}  A^{(m)}_{i,j} Q^h_{x^i,x^j} + \sum_{i=0}^{2^m-2} Q^{h}_{x^\star,x^i} + Q^{h}_{x^\star, x^{2^m-1}} + \rho^m Q^{h}_{ x^{2^m-1}, x^\star} \\
    \overset{\eqref{eq:silver-basic-for-gd}}{=} 
    & \sum_{i=0}^{2^m-2} \pi^{(m)}_i \left( h(x^\star)-h(x^{2^m-1}) - \left<\nabla h(x^i), x^\star - x^0 \right> \right)- \frac{L}{2} \left\| x^{2^m-1} - x^0 \right\|^{2} - \frac{\rho^{2m} +1}{2L} \left\| \nabla h(x^{2^m-1}) \right\|^{2} \\
    &  - \Big<\nabla h(x^{2^m-1}), x^\star - x^{2^m-1}\Big>  + (\rho^m-1) \big( h(x^{2^m-1}) - h(x^\star) \big)\\
    \hspace{2pt} \overset{(*)}{=} \hspace{2pt}&  L \Big< x^{2^m-1}-x^0, x^\star - x^0\Big> + \frac{L}{2} \left\| x^0-x^\star \right\|^{2} - \frac{L}{2} \left\| x^0-x^\star \right\|^{2} - \frac{L}{2} \left\| x^{2^m-1} - x^0 \right\|^{2}\\
    &  - \frac{L}{2} \left\| \frac{\rho^m}{L}\nabla h(x^{2^m-1}) \right\|^{2} - \frac{1}{2L} \left\| \nabla h(x^{2^m-1}) \right\|^{2} -  \left<\nabla h(x^{2^m-1}), x^\star - x^{2^m-1} \right>\\
    \hspace{3.2pt} = \hspace{3.2pt} & - \frac{L}{2} \left\| \frac{\rho^m}{L}\nabla h(x^{2^m-1}) \right\|^{2} + \frac{L}{2} \left\| x^0-x^\star \right\|^{2} - \frac{L}{2} \left\| x^{2^m-1} - \frac{1}{L} \nabla h(x^{2^m-1}) -x^\star \right\|^{2},
  \end{align*}
  where we have used $\sum_{i=0}^{2^m-2} \pi^{(m)}_i = \rho^m-1$ and $-\sum_{i=0}^{2^m-2} \pi^{(m)}_i \nabla h(x^i) = L(x^{2^m-1}-x^0)$ in $(*)$. Reorganizing the terms yields the desired upper bound \eqref{silver_tight_gdnorm}. 
  
  Finally, the tightness of the bounds in \eqref{silver_tight_objval} and \eqref{silver_tight_gdnorm} can be easily confirmed by considering the Huber function as discussed in \cite[Section 3.1]{LG24}. We omit the details for brevity. 
\end{proof}

\subsection{Accelerated convergence rates for RPPA with silver stepsize schedules} \label{subsec:RPPA_with_silver}

We now return to the convex optimization problem \eqref{eq:merely_cvx_problem}. Let $\{ (x^k,z^k) \}_{k\ge 0}$ be the sequence generated by Algorithm \ref{rppa}. 
Consider the connection between RPPA and GD method as given in \eqref{alg:R-PPA-compact-form}, we can apply Theorem \ref{thm:silver_gd_tight} with
$h=f^\lambda$ and $L=1 / \lambda$ and obtain the following convergence results. 

\begin{thom}[Convergence rate of RPPA with silver stepsize schedule $\pi^{(m)}$] \label{thm:RPPA_silver_tight_basic}
  Let $x^{\star} \in \mathcal{X}^\star$ be arbitrarily fixed and $z^{\star}=\mathrm{Prox}_{\lambda f}(x^\star)=x^{\star}$. For any fixed $m\ge 1$, let $N=2^m-1$ and $\{ (x^k,z^k )\}_{k=0}^{N}$ be the sequence generated by Algorithm \ref{rppa} for solving \eqref{eq:merely_cvx_problem} with stepsize schedule $\boldsymbol{\alpha}=\pi^{(m)}$ from any starting point $x^0\in\mathbb{R}^d$. Then, we have
  \begin{align}
    f(z^N) + \frac{\lambda}{2} \left\| \nabla f ^{\lambda}(x^{N}) \right\|^{2} - f(x^\star) \le\frac{\left\| x^0-x^\star \right\|^{2}}{(4 \rho^m-2) \lambda} \text{~~and~~} 
    \left\| \nabla f^{\lambda}(x^N)\right\| \le\frac{\left\| x^0-x^\star \right\|}{\rho^m \lambda}.  \label{RPPA_silver_tight_gdnorm}
  \end{align}
\end{thom}

Note that $\rho^m = (N+1)^{\log_2 \rho} \approx (N+1)^{1.2716}$. Therefore, 
Theorem \ref{thm:RPPA_silver_tight_basic} has already provided an accelerated convergence rate for RPPA using the silver stepsize schedule. 
Moreover, the upper bound on $\left\| \nabla f^{\lambda}(x^N)\right\|$ given in \eqref{RPPA_silver_tight_gdnorm} is tight due to item (i) of Lemma \ref{RPPA_general_lower_bound}  and the fact that $\sum_{i=0}^{2^m-2}\pi^{(m)}_i=\rho^m-1$.
On the other hand, the first inequality in \eqref{RPPA_silver_tight_gdnorm}
also yields an accelerated complexity bound for the function value residual $f(z^N) - f(x^{\star})$.
However, this bound is no longer tight. Deriving a tight upper bound for $(f(z^N) - f(x^{\star})) / \left\| x^{0}-x^{\star} \right\|^{2}$ from the first inequality in \eqref{RPPA_silver_tight_gdnorm} remains challenging.
This challenge motivates us to construct two modified silver stepsize schedules: the ``right silver" stepsize schedule $\overline{\pi}^{(m)}$ and the ``left silver" stepsize schedule $\underline{\pi}^{(m)}$. 
These two modified silver stepsize schedules allow us to achieve accelerated and tight convergence rates for RPPA under the measures $(f(z^N)-f(x^{\star})) / \left\| x^0-x^{\star} \right\|^{2}$ and $\left\| \nabla f^{\lambda}(x^N) \right\|^{2} / (f(x^0)-f(x^{\star}))$, respectively.

Our ``right silver'' stepsize schedule $\overline{\pi}^{(m)}$ is defined by extending the silver stepsize schedule with an additional element on the right, i.e., 
\begin{equation} \label{def:right_silver}
  \overline{\pi}^{(0)} := [\gamma_0]  \text{~~and~~}   \overline{\pi}^{(m)} := [\pi^{(m)}, \gamma_{m}], \quad \forall m\ge 1,
\end{equation}
where $\gamma_m = \frac{1}{2} \left( 1+\sqrt{1+4\rho^{m}} \right) \geq 1$ for any $m\ge 0$ is a solution to
\begin{equation} \label{def:gammak}
  \gamma_m^{2} = \gamma_m + \rho^{m}.
\end{equation}
On the other hand, our ``left silver'' stepsize schedule $\underline{\pi}^{(m)}$ 
is defined as
\begin{equation} \label{def:left_silver}
  \underline{\pi}^{(0)}:=[\gamma_0]   \text{~~and~~}   \underline{\pi}^{(m)} := [\gamma_m, \pi^{(m)}], \quad \forall m\ge 1.
\end{equation}
Define $T_m := 1 + \gamma_m +\sum_{i=0}^{2^m-2} \pi_i^{(m)}$ for $m\ge 1$. Then, it follows from \eqref{def:gammak} and $\sum_{i=0}^{2^m-2} \pi_i^{(m)}=\rho^m-1$ that
\begin{equation}
T_m = 1 + \gamma_m +\sum_{i=0}^{2^m-2} \pi_i^{(m)} = \gamma_m+\rho^m \overset{\eqref{def:gammak}}{=} \gamma_m^{2} = \frac{1}{2} + \frac{1}{2} \sqrt{1+4\rho^{m}} + \rho^{m}. \label{def:Tm}
\end{equation}
For $N = 2^m$, we have $T_m = \frac{1}{2} + \frac{1}{2}\sqrt{1 + 4 N^{\log_2 \rho}} + N^{\log_2 \rho} \sim N^{\log_2 \rho}$.

To prove convergence results of RPPA with the silver stepsize schedules $\overline{\pi}^{(m)}$ and $\underline{\pi}^{(m)}$, we need to adapt Lemma \ref{lem:silver_basic_lem} from GD method to RPPA (see Lemma \ref{lem:RPPA_silver_basic_lem}). 
This adaptation is based on the relations \eqref{eq:gd_mor_envlp} and \eqref{eq:fval_mor_envlp}. 
Before proceeding, we introduce the notation $P^f_{y,z^j}$, which is crucial to our proof.

\begin{definition}
  Given $N \in \mathbb{N}_+$, $\boldsymbol{\alpha} \in \mathbb{R}^{N}_{++}$ and $x^0\in\mathbb{R}^d$. 
  Suppose that $\{ (x^k,z^k) \}_{k=0}^N$ is the sequence generated by 
  Algorithm \ref{rppa} with stepsize schedule $\boldsymbol{\alpha}$ for solving \eqref{eq:merely_cvx_problem}.
  Let $x^\star\in\mathcal{X}^\star$ be any minimizer of $f$ and $z^{\star}=\mathrm{Prox}_{\lambda f}(x^\star)=x^{\star}$. 
  For any $y \in \mathbb{R}^{d}$ and $j \in \{ \star, 0, 1, \ldots, N \}$, define
\begin{equation}
  P^f_{y,z^j}:= f(y) - f(z^j) - \left<\nabla f^{\lambda}(x^j), y-z^j \right>.  \label{def:Pf_yz}
\end{equation}
  From \eqref{eq:gd_mor_envlp}, we have $\nabla f^{\lambda}(x^j) \in \partial f(z^{j})$ and thus $P^f_{y,z^j}\ge 0$. 
\end{definition}

Next, 
we show that $Q^{f^{\lambda}}_{x^i,x^j} \equiv P_{z^{i},z^{j}}^f$ for any $i,j \in  \{ \star,0,1, \ldots ,N \}$.

\begin{lemma}\label{lem:basic_equivalence}
  Given any fixed $N \in \mathbb{N}_+$, $\boldsymbol{\alpha} \in \mathbb{R}^{N}_{++}$ and $x^0\in\mathbb{R}^d$. Suppose that $\{ (x^k,z^k) \}_{k=0}^N$ is the sequence generated by Algorithm \ref{rppa} with stepsize schedule $\boldsymbol{\alpha}$ for solving \eqref{eq:merely_cvx_problem}. 
  Let $x^{\star}\in\mathcal{X}^{\star}$ be any minimizer of $f$ and $z^{\star}=\mathrm{Prox}_{\lambda f}(x^\star)=x^{\star}$. Then, we have
  \begin{equation} \label{eq:Q-equiv-P}
    Q^{f^{\lambda}}_{x^i,x^j} \equiv P_{z^{i},z^{j}}^f, \quad \forall i,j \in  \{ \star,0,1, \ldots ,N \}.
  \end{equation}
\end{lemma}
\begin{proof}
  Recall that $f^{\lambda}$ is convex and $(1 / \lambda)$-smooth. Further considering the notation defined in \eqref{def:Qh_xy} and \eqref{def:Pf_yz} and the relations  \eqref{eq:gd_mor_envlp} and \eqref{eq:fval_mor_envlp}, we obtain
  \begin{align*}
    Q^{f^{\lambda}}_{x^i,x^j} 
    \hspace{4.3pt} \overset{\eqref{def:Qh_xy}}{=} \hspace{4.3pt} &  
    f^{\lambda}(x^i) - f^{\lambda}(x^j) 
    - \left<\nabla f^{\lambda}(x^{j}), x^{i} - x^{j} \right> - \frac{\lambda}{2} \left\| \nabla f^{\lambda}(z^{i}) - \nabla f^{\lambda}(z^{j}) \right\|^{2} \\
    \overset{\eqref{eq:gd_mor_envlp}, \eqref{eq:fval_mor_envlp}}{=} & f(z^{i}) - f(z^{j}) - \left<\nabla f^{\lambda}(x^{j}), z^{i} - z^{j} + \lambda \left( \nabla f^{\lambda}(z^{i}) - \nabla f^{\lambda}(z^{j}) \right)  \right> \\
    & + \frac{\lambda}{2} \left<\nabla f^{\lambda}(x^{i})- \nabla f^{\lambda}(x^{j}), \nabla f^{\lambda}(x^{i}) + \nabla f^{\lambda}(x^{j}) \right> - \frac{\lambda}{2} \left\| \nabla f^{\lambda}(z^{i}) - \nabla f^{\lambda}(z^{j}) \right\|^{2} \\
    =\hspace{7.5pt} & f(z^{i}) - f(z^{j}) - \left<\nabla f^{\lambda}(x^{j}), z^{i} - z^{j}   \right> = P^{f}_{z^{i},z^{j}},
  \end{align*} 
  for any $i,j \in \{ \star, 0, 1, \ldots ,N \}$. This completes the proof.
\end{proof}

By invoking $h = f^{\lambda}$ in Lemma \ref{lem:silver_basic_lem}, we can derive the following lemma, which plays a key role in proving the convergence results for RPPA with the modified silver stepsize schedules $\overline{\pi}^{(m)}$ and $\underline{\pi}^{(m)}$.

\begin{lemma}[Adapted from Lemma \ref{lem:silver_basic_lem}] \label{lem:RPPA_silver_basic_lem}
  For any fixed $m\ge 1$, let $N=2^{m}-1$ and $\{ (x^{k}, z^{k}) \}_{k=0}^N$ be the sequence generated by Algorithm \ref{rppa} for solving \eqref{eq:merely_cvx_problem} with the silver stepsize schedule $\boldsymbol{\alpha}=\pi^{(m)}$ from an arbitrarily starting point $x^0\in\mathbb{R}^d$. Then, we have
  \begin{equation} \label{eq:RPPA-basic-pep-eq}
    \begin{aligned}
      0\le \sum_{i,j \in \{ 0,1, \ldots ,2^m-1 \}} A^{(m)}_{i,j} P_{z^{i},z^{j}}^f = & \sum_{i=0}^{2^m-2} \alpha_i \left( f(z^{i}) - f(z^{2^m-1}) +\left< \nabla f^{\lambda}(x^{i}), x^{0}-z^{i} \right> \right) \\
      & - \frac{(\rho^{2m}-1) \lambda}{2} \left\| \nabla f^{\lambda}(x^{2^m-1}) \right\|^{2} - \frac{\lambda}{2} \left\| \sum_{i=0}^{2^m-2} \alpha_i \nabla f^{\lambda}(x^{i}) \right\|^{2}, \\
    \end{aligned}
  \end{equation}
  where $A^{(m)} \in \mathbb{R}^{2^m\times 2^m}$,  
  whose components are indexed by
   $(i,j)\in \{0,1, \ldots, 2^m-1\}^2$, 
   is defined in \eqref{def:A_mat_recur}.
\end{lemma}
\begin{proof}
  Due to the connection \eqref{alg:R-PPA-compact-form}, we can invoke $h=f^{\lambda}$ and $L=1 / \lambda$ in Lemma \ref{lem:silver_basic_lem} and then deduce
  \begin{align*}
    0\le  \sum_{i,j \in \{ 0,1, \ldots ,2^m - 1 \}} A^{(m)}_{i,j} Q^{f^{\lambda}}_{x^{i}, x^{j}} =& \sum_{i=0}^{2^m-2} \alpha_i\left( f^{\lambda}(x^{i}) - f^{\lambda}(x^{2^m-1}) + \left<\nabla f^{\lambda}(x^{i}),x^{0}-x^{i} \right> + \frac{\lambda}{2} \left\| \nabla f^{\lambda}(x^{i}) \right\|^{2} \right) \\ 
    & - \frac{1}{2\lambda} \left\| x^{2^m-1} - x^{0} \right\|^{2} - \frac{\rho^m (\rho^m-1) \lambda}{2} \left\| \nabla f^{\lambda}(x^{2^m-1}) \right\|^{2}.
  \end{align*}
  Then, by leveraging Lemma \ref{lem:basic_equivalence} and the relations \eqref{eq:gd_mor_envlp} and \eqref{eq:fval_mor_envlp}, we derive
  \begin{align*}
    0 \hspace{1pt} \le \hspace{1pt} & \sum_{i,j \in \{ 0,1, \ldots ,2^m - 1 \}} A^{(m)}_{i,j} P_{z^{i},z^{j}}^f \overset{\eqref{eq:Q-equiv-P}}{=} \sum_{i,j \in \{ 0,1, \ldots ,2^m - 1 \}} A^{(m)}_{i,j} Q^{f^{\lambda}}_{x^{i}, x^{j}} \\
    \overset{\eqref{eq:fval_mor_envlp}}{=} & \sum_{i=0}^{2^m-2} \alpha_i\Big( f(z^{i}) + \frac{\lambda}{2} \left\| \nabla f^{\lambda}(x^{i}) \right\|^{2} - f(z^{2^m-1}) - \frac{\lambda}{2} \left\| \nabla f^{\lambda}(x^{2^m-1}) \right\|^{2} + \left<\nabla f^{\lambda}(x^{i}),x^{0}-x^{i} \right> \\ 
    & \qquad \qquad  + \frac{\lambda}{2} \left\| \nabla f^{\lambda}(x^{i}) \right\|^{2} \Big)  - \frac{1}{2\lambda} \left\|x^{2^m-1} - x^{0} \right\|^{2}- \frac{(\rho^{2m}- \rho^m) \lambda}{2} \left\| \nabla f^{\lambda}(x^{2^m-1}) \right\|^{2}  \\
    \overset{(*)}{=} &   \sum_{i=0}^{2^m-2} \alpha_i\left( f(z^{i})  - f(z^{2^m-1}) + \left<\nabla f^{\lambda}(x^{i}),x^{0}-x^{i} \right> + \lambda\left\| \nabla f^{\lambda}(x^{i}) \right\|^{2} \right)\\
    & \qquad - \frac{\lambda}{2} \left\|\sum_{i=0}^{2^m-2} \alpha_i \nabla f^{\lambda}(x^{i}) \right\|^{2}- \frac{(\rho^{2m}- 1) \lambda}{2} \left\| \nabla f^{\lambda}(x^{2^m-1}) \right\|^{2} \\
    \overset{\eqref{eq:gd_mor_envlp}}{=} & \sum_{i=0}^{2^m-2} \alpha_i \left( f(z^{i}) - f(z^{2^m-1}) +\left< \nabla f^{\lambda}(x^{i}), x^{0}-z^{i} \right> \right) \\
    & \qquad - \frac{\lambda}{2} \left\| \sum_{i=0}^{2^m-2} \alpha_i \nabla f^{\lambda}(x^{i}) \right\|^{2}  - \frac{(\rho^{2m}-1) \lambda}{2} \left\| \nabla f^{\lambda}(x^{2^m-1}) \right\|^{2},
  \end{align*}
  where we have used $\sum_{i=0}^{2^m-2} \alpha_i = \sum_{i=0}^{2^m-2} \pi^{(m)}_i = \rho^m-1$ and $x^{2^m-1}-x^0=-\sum_{i=0}^{2^m-2} \lambda \alpha_i \nabla f^{\lambda}(x^i)$ in $(*)$. This completes the proof.
\end{proof}

The following technical lemma will be used in subsequent analysis. Although it may appear complex initially, its verification is elementary as it comprises equations.

\begin{lemma}\label{lem:simplify-terms}
  For any vectors $a,b,c,d \in \mathbb{R}^{d}$ and positive scalars $r,s \in \mathbb{R}_{++}$, we have
  \begin{align}
    &\left<b,a \right> - \frac{1}{2} (r^2-1) \left\| c \right\|^{2} - \frac{1}{2} \left\| b \right\|^{2} + (r+1) \left<c, a- b- c \right> + 2r \left<d, c-d-sc \right> + 2 s \left<d, a-b-sc-d \right> \nonumber\\
    & = \frac{1}{2} \big( \left\| a \right\|^{2} - \left\| a-b-(1+r)c-2sd \right\|^{2} \big) + 2 \left( s^{2} - s - r \right) \left<d,d-c \right> \label{eq:simplify-1}
  \end{align}
  and
  \begin{align}
    &\frac{s+r}{r^{2}} \big( \left<b, (1-s)c + d \right> - (r^{2}-1)\left\| d \right\|^{2}  - \left<d, b+(s-1)c+d \right>\big) - \left\| c \right\|^{2} + \frac{s}{r} \left<c,b+(s-1)c+d \right>\nonumber\\
    & = - (s+r) \left\| d \right\|^{2} + \frac{s^{2}-s-r}{r^{2}} \left<c, -b+rc-d \right>. \label{eq:simplify-2}
  \end{align}
\end{lemma}

Then for RPPA with the right silver stepsize schedule $\overline{\pi}^{(m)}$, we establish an upper bound for the measure $(f(z^N) - f(x^\star)) / \left\| x^0-x^{\star} \right\|^{2}$ in the following theorem.

\begin{thom}[Tight convergence rate of RPPA with the right silver stepsize schedule  $\overline{\pi}^{(m)}$]\label{thm:objval-vs-ptnorm}
  For any fixed $m\ge 0$, let $N=2^m$ and $\{ (x^{k},z^{k}) \}_{k=0}^N$ be the sequence generated by Algorithm \ref{rppa} with stepsize schedule $\boldsymbol{\alpha} = \overline{\pi}^{(m)}$ for solving \eqref{eq:merely_cvx_problem} from any starting point $x^0 \in \mathbb{R}^{d}$. Let $x^{\star} \in \mathcal{X}^{\star}$ be any minimizer of $f$ and $z^{\star}=\mathrm{Prox}_{\lambda f}(x^\star)=x^{\star}$. Then, there holds
  $$
  f(z^N) - f(x^\star) \le \frac{1}{2T_m} \frac{1}{2\lambda}\left\| x^0-x^\star \right\|^{2} \sim \frac{1}{2 N ^{\log_2 \rho}} \frac{1}{2\lambda}\left\| x^0-x^\star \right\|^{2},
  $$
  where $T_m$ is defined in \eqref{def:Tm}. 
  Furthermore, this bound is tight.
\end{thom}
\begin{proof}
  For $m=0$, we have $N=1$ and $\boldsymbol{\alpha}=\overline{\pi}^{(0)} = [\gamma_0]$. 
  Note that $0\le P_{z^{\star},z^{i}}^{f} =f(z^{\star}) - f(z^{i}) - \left< \nabla f^{\lambda}(x^{i}), z^{\star} - z^{i} \right>$, $0\le P_{z^{0},z^{1}}^{f} = f(z^{0})- f(z^{1}) - \left<\nabla f^{\lambda}(x^{1}), z^{0}-z^{1} \right>$, and $T_0 = 1+\gamma_0$. We thus have
  \begin{align}
    0 \hspace{7.8pt} \le \hspace{7.8pt} &  2  P^f_{z^{\star},z^{0}} + 2 P^f_{z^{0},z^{1}} + 2 \gamma_0 P^f_{z^{\star},z^{1}} \nonumber\\
    {\overset{\eqref{eq:gd_mor_envlp}, \eqref{alg:R-PPA-compact-form}}{=}} & 2 T_0 \left( f(z^{\star}) -  f(z^{1})\right) +  2 \left<\nabla f^{\lambda}(x^{0}), x^{0}-z^{\star} - \lambda\nabla f^{\lambda}(x^{0})\right> \nonumber\\
    & + 2 \left<\nabla f^{\lambda}(x^{1}), -\gamma_0 \lambda \nabla f^{\lambda}(x^{0})- \lambda \left( \nabla f^{\lambda}(x^{1}) - \nabla f^{\lambda}(x^{0}) \right)  \right> \nonumber \\
    & + 2 \gamma_0 \left<\nabla f^{\lambda}(x^{1}), x^{0}-z^{\star} - \gamma_0 \lambda \nabla f^{\lambda}(x^{0}) - \lambda \nabla f^{\lambda}(x^{1}) \right>, \nonumber \\
    {\overset{\eqref{eq:simplify-1}}{=}} \hspace{4.3pt} & 2 T_0 \left( f(z^{\star}) -  f(z^{1})\right)  +  \frac{1}{2 \lambda} \big( \left\| x^{0}-z^{\star} \right\|^{2} - \left\| x^{0}-z^{\star} - 2 \lambda \nabla f^{\lambda}(x^{0})-2\gamma_0 \lambda \nabla f^{\lambda}(x^{1})\right\|^{2} \big), 
    \label{eq:k0-fval-dist}
  \end{align}
  where in the second ``$=$" we have utilized \eqref{eq:simplify-1} with 
  $a = \lambda^{-1 /2} (x^{0}-z^{\star})$, $b={\bf 0}$, $c=\lambda^{1 /2} \nabla f^{\lambda}(x^{0})$, $d=\lambda^{1 /2} \nabla f^{\lambda}(x^{1})$, $r=1$, $s=\gamma_0$, 
  and the fact that $\gamma_0^{2}- \gamma_0-1 = 0$. Hence, \eqref{eq:k0-fval-dist} implies that the desired bound holds for $m=0$.

  For $m\ge 1$, we have $N=2^m$ and $\boldsymbol{\alpha} = \overline{\pi}^{(m)} = [\pi^{(m)},\gamma_m]$. 
  Since the first $2^m-1$ stepsizes are identical to the silver stepsize schedule $\pi^{(m)}$, the inequality \eqref{eq:RPPA-basic-pep-eq} in Lemma \ref{lem:RPPA_silver_basic_lem} still holds. Then, by noting $0\le P_{z^{\star},z^{i}}^f = f(z^{\star}) - f(z^{i}) - \left<\nabla f^{\lambda}(z^{i}), z^{\star} - z^{i} \right>$, $0\le P_{z^{2^m-1}, z^{2^m}}^f = f(z^{2^m-1}) - f(z^{2^m}) - \left<\nabla f^{\lambda}(z^{2^m}), z^{2^m-1} - z^{2^m} \right>$ and $T_m = 1 + \gamma_m + \sum_{i=0}^{2^m-2} \alpha_i = \gamma_m + \rho^m$, we derive
  \begin{align} 
    0\hspace{7.8pt} \le \hspace{7.8pt}
    & \sum_{i,j \in \{ 0,1, \ldots , 2^m -1\}} A^{(m)}_{i,j} P_{z^i,z^j}^f + \sum_{i=0}^{2^m-2} \alpha_i P^f_{z^\star, z^i} + (\rho^m+1) P^f_{z^\star, z^{2^m-1}}  + 2\rho^m P^f_{z^{2^m-1},z^{2^m}} + 2\gamma_m P^f_{z^\star, z^{2^m}}\nonumber\\
    \overset{\eqref{eq:RPPA-basic-pep-eq}}{=} \hspace{4.3pt} & \sum_{i=0}^{2^m-2} \alpha_i \left( f(z^\star)-f(z^{2^m-1}) + \left< \nabla f^{\lambda}(x^{i}), x^0-z^\star \right> \right) - \frac{(\rho^{2m}-1)\lambda}{2} \left\| \nabla f^{\lambda}(x^{2^m-1})\right\|^{2}  \nonumber\\
    & 
    \qquad +(\rho^m+1)\left( f(z^{\star}) - f(z^{2^m-1}) - \left<\nabla f^{\lambda}(x^{2^m-1}), z^\star-z^{2^m-1} \right> \right)  \nonumber\\
    &     \qquad + 2 \rho^m \left( f(z^{2^m-1}) - f(z^{2^m}) - \left<\nabla f^{\lambda}(x^{2^m}), z^{2^m-1}-z^{2^m} \right> \right) 
        - \frac{\lambda}{ 2} \left\| \sum_{i=0}^{2^m-2} \alpha_i \nabla f^{\lambda}(x^{i})\right\|^{2} \nonumber\\
    &     \qquad +2 \gamma_m \left( f(z^\star) - f(z^{2^m}) - \left<\nabla f^{\lambda}(x^{2^m}), z^\star-z^{2^m} \right> \right) \nonumber\\
    {\overset{\eqref{eq:gd_mor_envlp}, \eqref{alg:R-PPA-compact-form}}{=}} & 2T_m \left( f(z^{\star}) - f(z^{2^m}) \right)  + \left<\sum_{i=0}^{2^m-2} \alpha_i \nabla f^{\lambda}(x^{i}) , x^0 - z^\star \right> - \frac{(\rho^{2m}-1)\lambda}{2} \left\| \nabla f^{\lambda}(x^{2^m-1}) \right\|^{2}  \nonumber\\
    & + (\rho^m + 1) \left<\nabla f^{\lambda}(x^{2^m-1}), x^0 - z^\star - \lambda \sum_{i=0}^{2^m-2}\alpha_i \nabla f^{\lambda}(x^{i}) -\lambda \nabla f^{\lambda}(x^{2^m-1}) \right>  \nonumber\\
    & + 2\rho^m \lambda \left< \nabla f^{\lambda}(x^{2^m}),  \nabla f^{\lambda}(x^{2^m-1}) -\nabla f^{\lambda}(x^{2^m}) - \gamma_m \nabla f^{\lambda}(x^{2^m-1}) \right> 
    - \frac{\lambda}{ 2} \left\|\sum_{i=0}^{2^m-2} \alpha_i \nabla f^{\lambda}(x^{i})\right\|^{2} \nonumber\\
    & + 2\gamma_m \left<\nabla f^{\lambda}(x^{2^m}), x^0-z^\star -  \lambda \sum_{i=0}^{2^m-2}\alpha_i \nabla f^{\lambda}(x^{i}) - \gamma_m \lambda \nabla f^{\lambda}(x^{2^m-1}) - \lambda \nabla f^{\lambda}(x^{2^m}) \right>  \nonumber\\
    {\overset{ \eqref{eq:simplify-1}}{=}} \hspace{4.3pt} & - \frac{1}{2} \left\| \lambda^{-1/2}(x^0-z^\star)-\lambda^{1 /2} \sum_{i=0}^{2^m-2}\alpha_i \nabla f^{\lambda}(x^{i}) - (1+\rho^m) \lambda ^{1 /2} \nabla f^{\lambda}(x^{2^m-1})-2\gamma_m \lambda ^{1 /2} \nabla f^{\lambda}(x^{2^m}) \right\|^{2} \nonumber \\
    & + 2T_m\left( f(z^{\star}) - f(z^{2^m}) \right)+ \frac{1}{2 \lambda} \left\| x^0-z^\star \right\|^{2}, \label{eq:kgeq1-fval-dist}
  \end{align}
  where in the last ``$=$" we have utilized \eqref{eq:simplify-1} with $a=\lambda^{-1 /2} (x^{0}-z^{\star})$, $b=\lambda^{1 /2} \sum_{i=0}^{2^m-2}\alpha_i \nabla f^{\lambda}(x^{i})$, $c=\lambda ^{1 /2} \nabla f^{\lambda}(x^{2^m-1})$, $d=\lambda ^{1 /2} \nabla f^{\lambda}(x^{2^m})$, $r=\rho^m$,  $s=\gamma_m$, and the fact that $\gamma_m^{2}- \gamma_m-\rho^m = 0$.
  Hence, \eqref{eq:kgeq1-fval-dist} implies that the desired bound also holds for $m\ge 1$. Finally, the tightness of the bound can be verified by using item (ii) of Lemma \ref{RPPA_general_lower_bound} and the fact that $T_m = 1 + \sum_{i=0}^{2^m-1}\alpha_i$.
\end{proof}

Our final result concerns RPPA with the left silver stepsize schedule $\underline{\pi}^{(m)}$. Here, we establish a tight upper bound for the measure $\left\| \nabla f^{\lambda}(x^{N}) \right\|^{2} / (f(x^0) - f(x^\star) )$. This measure is especially relevant in overparameterized models, which are frequently encountered in machine learning applications.
In such models, $f(x^{\star})$ is often a known constant, such as $0$. In these cases, the initial function value residual is directly computable, whereas the initial distance to a minimizer can be difficult to estimate. 

\begin{thom}[Tight convergence rate of RPPA with the left silver stepsize schedule $\underline{\pi}^{(m)}$]\label{thm:subgdnorm-vs-objval}
  For any fixed $m\ge 0$, let $N=2^m$ and $\{ (x^{k},z^{k}) \}_{k=0}^N$ be the sequence generated by Algorithm \ref{rppa} with stepsize schedule $\boldsymbol{\alpha} = \underline{\pi}^{(m)}$ for solving \eqref{eq:merely_cvx_problem} from any starting point 
  $x^0\in \mathrm{dom}(f)$.  
  Let $x^{\star}\in \mathcal{X}^{\star}$ be any minimizer of $f$ and $z^{\star}=\mathrm{Prox}_{\lambda f}(x^{\star})=x^{\star}$. Then, there holds
  $$
  \frac{\lambda}{2}\left\| \nabla f^{\lambda}(x^{N}) \right\|^{2} \le \frac{1}{2T_m} \left( f(x^0) - f(x^\star) \right) \sim  \frac{1}{2 N ^{\log_2 \rho}} \left( f(x^0) - f(x^\star) \right),
  $$
  where $T_m$ is defined in \eqref{def:Tm}. 
  Furthermore, this bound is tight.
\end{thom}
\begin{proof}
  For $m=0$, we have $N=1$ and $\boldsymbol{\alpha}= [\gamma_0]$. Note that $0\le P_{x^{0},z^{0}}^{f} 
  = f(x^{0}) - f(z^{0}) - \lambda \left\| \nabla f^{\lambda}(x^{0}) \right\|^{2}$, 
  $0\le P_{z^{1},z^{\star}}^{f} = f(z^{1})- f(z^{\star})$, 
  $0\le  P^f_{z^{1},z^{0}} = f(z^{1})- f(z^{0}) - \left<\nabla f^{\lambda}(x^{0}), z^{1}-z^{0} \right> $, 
  $0\le  P^f_{z^{0},z^{1}} = f(z^{0})- f(z^{1}) - \left<\nabla f^{\lambda}(x^{1}), z^{0}-z^{1} \right>$, and $T_0 = 1+\gamma_0$. We thus have
  \begin{align}
    0 \hspace{7.8pt} \le \hspace{7.8pt} & T_0 P^{f}_{z^{0}, z^{1}} + \gamma_0 P^{f}_{z^{1},z^{0}} + P^{f}_{x^{0},z^{0}} + P^{f}_{z^{1},z^{\star}}  \nonumber\\
    {\overset{\eqref{eq:gd_mor_envlp},\eqref{alg:R-PPA-compact-form}}{=}} & f(x^{0}) - f(z^{\star}) - T_0 \lambda\left< \nabla f^{\lambda}(x^{1}), (\gamma_0-1) \nabla f^{\lambda}(x^{0}) + \nabla f^{\lambda}(x^{1}) \right> \nonumber\\
    & + \gamma_0 \lambda \left< \nabla f^{\lambda}(x^{0}), (\gamma_0-1) \nabla f^{\lambda}(x^{0}) + \nabla f^{\lambda}(x^{1})\right> - \lambda \left\| \nabla f^{\lambda}(x^{0}) \right\|^{2} \nonumber \\
    \overset{\eqref{eq:simplify-2}}{=} \hspace{4.3pt} & f(x^{0}) - f(z^{\star}) - T_0 \lambda \left\| \nabla f^{\lambda}(x^{1}) \right\|^{2} + (\gamma_0^{2}-\gamma_0-1) \lambda\left<\nabla f^{\lambda}(x^{0}),  \nabla f^{\lambda}(x^{0})- \nabla f^{\lambda}(x^{1})\right>, \label{eq:k0-gd-fval}
  \end{align}
  where in the second ``$=$" we have utilized \eqref{eq:simplify-2} with $a=b={\bf 0}$, $c = \lambda ^{1 /2} \nabla f^{\lambda}(x^{0})$, $d=\lambda ^{1 /2} \nabla f^{\lambda}(x^{1})$, $r=1$ and $s=\gamma_0$. Since  $\gamma_0^{2}-\gamma_0-1=0$, 
  \eqref{eq:k0-gd-fval} implies that the desired bound holds for $m=0$.

  For $m\ge 1$, we have $N=2^m$ and $\boldsymbol{\alpha}=\underline{\pi}^{(m)} = [\gamma_m,\pi^{(m)}]$. 
  We define the multiplier matrix $C^{(m)} := \frac{2 T_m}{\rho^{2m}} A^{(m)} \in  \mathbb{R}^{2^m \times 2^m}$, whose components are indexed by $(i,j)\in\{ 1,2, \ldots ,2^m \}^2$. 
  Since that the last $2^m-1$ components of $\underline{\pi}^{(m)}$ are identical to the silver stepsize schedule $\pi^{(m)}$, the sequence $\{ x^{1} \} \cup \{ z^{k} \}_{k=1}^{2^m}$ is generated by Algorithm \ref{rppa} with stepsize schedule $\pi^{(m)}$ from starting point $x^{1}$. 
  Thus, by Lemma \ref{lem:RPPA_silver_basic_lem}, we have
  \begin{align}
    0\le \sum_{i,j \in \{ 1, 2, \ldots ,2^m \}}C^{(m)}_{i,j} P^{f}_{z^{i},z^{j}} =& \frac{2T_m}{\rho^{2m}} \Bigg[ \sum_{i=1}^{2^m-1} \alpha_i \left( f(z^{i}) - f(z^{2^m}) +\left< \nabla f^{\lambda}(x^{i}), x^{1}-z^{i} \right> \right)  \nonumber\\
    &- \frac{(\rho^{2m}-1) \lambda}{2} \left\| \nabla f^{\lambda}(x^{2^m}) \right\|^{2} - \frac{\lambda}{2} \left\| \sum_{i=1}^{2^m-1} \alpha_i \nabla f^{\lambda}(x^{i}) \right\|^{2} \Bigg]. \label{eq:ck-mat-id}
  \end{align}
  Note that $0\le P_{x^{0},z^{0}}^{f} 
  = f(x^{0}) - f(z^{0}) - \lambda \left\| \nabla f^{\lambda}(x^{0}) \right\|^{2}$, $0\le P_{z^{2^m},z^{\star}}^{f} = f(z^{2^m})- f(z^{\star})$, $0\le  P^f_{z^{i},z^{0}} = f(z^{i})- f(z^{0}) - \left<\nabla f^{\lambda}(x^{0}), z^{i}-z^{0} \right> $, $0\le  P^f_{z^{0},z^{i}} = f(z^{0})- f(z^{i}) - \left<\nabla f^{\lambda}(x^{i}), z^{0}-z^{i} \right> $, and $T_m = 1+\gamma_m + \sum_{i=1}^{2^m-1}\alpha_i=\gamma_m + \rho^m$. We thus have
  \begin{align}
    0 \hspace{7.8pt} \le \hspace{7.8pt} & \hspace{-15pt} \sum_{i,j \in \{ 1, \ldots ,2^m \}}C^{(m)}_{i,j} P^{f}_{z^{i},z^{j}} + \frac{T_m}{\rho^{2m}} \left[ \sum_{i=1}^{2^m}\alpha_i \left( P^{f}_{z^{0},z^{i}}+ P^{f}_{z^{2^m},z^{i}} \right)  + P^{f}_{z^{0}, z^{2^m}}  \right] + \frac{\gamma_m}{\rho^m}P^{f}_{z^{2^m}, z^{0}} + P^{f}_{x^{0},z^{0}} + P^{f}_{z^{2^m},z^{\star}} \nonumber \\
    \overset{\eqref{eq:ck-mat-id}}{=} \hspace{4.3pt}& 
    \frac{T_m}{\rho^{2m}} \left(1 + \sum_{i=1}^{2^m-1} \alpha_i \right)  \left( f(z^{0}) - f(z^{2^m}) \right)  + \frac{\gamma_m}{\rho^m} \left(  f(z^{2^m}) - f(z^{0})\right) +   f(x^{0}) - f(z^{0})  + f(z^{2^m})- f(z^{\star}) \nonumber \\
    & + \frac{T_m}{\rho^{2m}} \Bigg( \sum_{i=1}^{2^m} \alpha_i \left<\nabla f^{\lambda}(x^{i}), 2x^{1}-z^{0}-z^{2^m} \right> - (\rho^{2m}-1) \lambda \left\| \nabla f^{\lambda}(x^{2^m}) \right\|^{2} - \lambda \left\| \sum_{i=1}^{2^m-1} \alpha_i \nabla f^{\lambda}(x^{i}) \right\|^{2}  \nonumber \\
    & + \left<\nabla f^{\lambda}(x^{2^m}), z^{2^m} - z^{0} \right> \Bigg)  + \frac{\gamma_m}{\rho^m}  \left<\nabla f^{\lambda}(x^{0}), z^{0} - z^{2^m} \right> - \lambda \left\| \nabla f^{\lambda}(x^{0}) \right\|^{2} \nonumber \\
     {\overset{\eqref{eq:gd_mor_envlp},\eqref{alg:R-PPA-compact-form}}{=}} & f(x^{0}) - f(z^{\star}) + \frac{T_m}{\rho^{2m}} \Bigg( \lambda \left< \sum_{i=1}^{2^m}\alpha_i\nabla f^{\lambda}(x^{i}) ,   (1-\gamma_m) \nabla f^{\lambda}(x^{0}) + \nabla f^{\lambda}(x^{2^m}) \right>  \nonumber \\
    &   - \lambda\left<\nabla f^{\lambda}(x^{2^m}),  \sum_{i=1}^{2^m}\alpha_i\nabla f^{\lambda}(x^{i}) + (\gamma_m-1)\nabla f^{\lambda}(x^{0}) +  \nabla f^{\lambda}(x^{2^m}) \right> \Bigg) - \lambda \left\| \nabla f^{\lambda}(x^{0}) \right\|^{2}  \nonumber \\
    & + \frac{\gamma_m}{\rho^m} \left<\nabla f^{\lambda}(x^{0}),  \sum_{i=1}^{2^m}\alpha_i\nabla f^{\lambda}(x^{i}) + (\gamma_m-1)\nabla f^{\lambda}(x^{0}) + \nabla f^{\lambda}(x^{2^m}) \right> - (\rho^{2m}-1) \lambda \left\| \nabla f^{\lambda}(x^{2^m}) \right\|^{2} \nonumber \\ 
    \overset{\eqref{eq:simplify-2}}{=} \hspace{4.3pt} & f(x^{0}) - f(z^{\star}) - T_m \lambda \left\| \nabla f^{\lambda}(x^{2^m}) \right\|^{2}, \label{eq:kgeq1-gd-fval}
  \end{align}
  where in the second ``$=$" we also used $\sum_{i=1}^{2^m-1}\alpha_i = \sum_{i=0}^{2^m-2}\pi^{(m)}_i = \rho^m-1$ and in the last ``$=$" we have utilized \eqref{eq:simplify-2} with $a={\bf 0}$, $b=\lambda^{1 /2}\sum_{i=1}^{2^m}\alpha_i\nabla f^{\lambda}(x^{i})$, $c=\lambda^{1 /2} \nabla f^{\lambda}(x^{0})$, $d=\lambda^{1 /2} \nabla f^{\lambda}(x^{2^m})$, $r=\rho^m$, $s=\gamma_m$, and the fact that $\gamma_m^{2}-\gamma_m-\rho^m=0$.
  Hence, \eqref{eq:kgeq1-gd-fval} implies that the desired bound also holds for $m\ge 1$. Finally, the tightness of the bound can be verified by using item (iii) of Lemma \ref{RPPA_general_lower_bound}  and the fact that $T_m = 1 + \sum_{i=0}^{2^m-1}\alpha_i$.
\end{proof}

\subsection{Discussion on the optimality of silver stepsize schedules}

Given a fixed stepsize schedule, we obtain a specific instance of the RPPA and can establish its tight convergence rate under certain performance measures. This naturally raises the following questions: What is the optimal stepsize schedule that yields the best tight convergence rate for RPPA? Are the silver stepsize schedules $\pi^{(m)}$, $\overline{\pi}^{(m)}$, and $\underline{\pi}^{(m)}$ that we previously analyzed optimal in this regard? In this section, we aim to address these questions.

First, we consider the optimality of the order of the accelerated convergence rate $O(1 / N^{\log_2 \rho})\approx O(1 / N^{1.2716})$, which we obtained for the RPPA using the (modified) silver stepsize schedules across different measures. 
This order of convergence has been conjectured optimal for GD method without momentum for solving smooth convex optimization problems \cite[Section 1.1]{AP23b} \cite[Remark 1]{GSW24}. Given the close relationship between RPPA and GD method, it is reasonable to conjecture that this order of rate is also optimal for RPPA without momentum, i.e. Algorithm \ref{rppa}.

Next, we examine the exact optimal stepsize schedules, that is, those that achieve the optimal rate even up to the constant.
For any fixed $N \in \mathbb{N}_+$ and $\boldsymbol{\alpha} \in \mathbb{R}^{N}_{++}$, the tight worst-case rate of Algorithm \ref{rppa} with the stepsize schedule $\boldsymbol{\alpha}$ under different measures can be computed numerically using the performance estimation approach proposed in \cite{DT14,THG17b,THG17c}. 
Additionally, for any fixed performance measure, we can conduct a numerical search for the optimal stepsize schedules that minimize the tight worst-case rate. 
The problem of finding the optimal stepsize schedules can be formulated as a nonconvex quadratically constrained quadratic programming problem and be solved numerically by a customized branch-and-bound method (abbreviated as BnB-PEP for later use), as described in \cite{DVR23}. 
In this paper, we have considered the following three measures: 
$\left\| \nabla f^{\lambda}(x^N) \right\| / \left\| x^{0}-x^{\star} \right\|$, 
$(f(z^{N}) - f(x^{\star}))  / \left\| x^{0} - x^{\star} \right\|^{2}$ 
and $\left\| \nabla f^{\lambda}(x^N) \right\|^{2} / (f(x^{0})-f(x^{\star}))$. In the following, we denote the numerically optimal stepsize schedules for these measures as $\boldsymbol{\alpha} ^{\rm gd,dist,\star}$, $\boldsymbol{\alpha} ^{\rm fval,dist,\star}$ and $\boldsymbol{\alpha} ^{\rm gd,fval,\star}$, respectively. 
In Table \ref{tab:Globally-optimal-stepsize-compare}, we compare $\boldsymbol{\alpha} ^{\rm gd,dist,\star}$ with
$\pi^{(m)}$, for which the measure $\left\| \nabla f^{\lambda}(x^N) \right\| / \left\| x^{0}-x^{\star} \right\|$ is used, and $\boldsymbol{\alpha} ^{\rm fval,dist,\star}$ with $\overline{\pi}^{(m)}$, for which the measure $(f(z^{N}) - f(x^{\star}))  / \left\| x^{0} - x^{\star} \right\|^{2}$  is employed. 

For the measure $\left\| \nabla f^{\lambda}(x^N) \right\| / \left\| x^{0}-x^{\star} \right\|$, we observe from  Table \ref{tab:Globally-optimal-stepsize-compare} that the silver stepsize schedule $\pi^{(m)}$ approximates the optimal stepsize schedule $\boldsymbol{\alpha} ^{\rm gd,dist,\star}$ found via BnB-PEP for $m=1,2,3$ (corresponding to $N=1,3,7$, respectively) quite well.
Therefore, we conjecture that the silver stepsize schedule $\pi^{(m)}$ is already optimal for RPPA for the measure $\left\| \nabla f^{\lambda}(x^N) \right\| / \left\| x^{0}-x^{\star} \right\|$.
For the measure $(f(z^{N}) - f(x^{\star})) / \left\| x^{0} - x^{\star} \right\|^{2}$, the constructed right silver stepsize schedule $\overline{\pi}^{(m)}$ approximates the optimal stepsize schedule found by BnB-PEP quite well when $m=0,2$ (corresponding to $N=1,4$, respectively). 
In Table \ref{tab:numeric-compare}, we present the tight worst-case upper bound of $(f(z^{N}) - f(x^{\star})) / \left\| x^{0} - x^{\star} \right\|^{2}$ corresponding to 
$\overline{\pi}^{(m)}$ and  $\boldsymbol{\alpha} ^{\rm fval,dist,\star}$ with $m=0,1,2,3$ (corresponding to $N=1,2,4,8$, respectively).
It can be seen from Table  \ref{tab:numeric-compare} that,
for $m=1$ (and $N=2$), our constructed schedule $\overline{\pi}^{(1)} \approx [1.414214, 2.132242]$ gives a tight worst-case bound of $(f(z^{N}) - f(x^{\star})) / \left\| x^{0} - x^{\star} \right\|^{2}$ to be $0.054988$. This value is larger than that corresponding to $\boldsymbol{\alpha} ^{\rm fval,dist,\star}$, which is $0.54900$. The discrepancy between the two is somehow nonnegligible.
A similar discrepancy is observed for $m=3$ (and $N=8$).
This indicates that $\overline{\pi}^{(m)}$ may be suboptimal, though it is quite close to the optimal stepsize schedule.

The right and left silver stepsize schedules $\overline{\pi}^{(m)}$ and $\underline{\pi}^{(m)}$ mirror each other and provide convergence rates that mirror the relationship between function value residual and subgradient norm (see Theorems \ref{thm:objval-vs-ptnorm} and \ref{thm:subgdnorm-vs-objval}). 
This symmetry has been observed in optimal gradient methods for solving smooth convex optimization problems \cite{GSW24, Kim+23}. It appears closely related to the H-duality theory presented in \cite{Kim+23}. 
Therein, it was shown that reversing the order of steps in methods using certain inductive proofs leads to a conversion between (sub)gradient norm and function value residual convergence.
Due to this fundamental connection, the left silver stepsize schedule $\underline{\pi}^{(m)}$ may also be suboptimal.


\section{Conclusion} \label{sec-conclusion}

In this paper, we have examined RPPA for solving convex optimization problems. Our analysis encompasses three types of relaxation schedules. For both the constant and the TV's relaxation schedules, we have established non-ergodic and tight $O(1 / N)$ convergence rate results by using the function value residual
or subgradient norm as the optimality measure. 
For the original silver relaxation schedule and two extensions, which permit $\alpha_{k}$ to exceed $2$, we have attained accelerated convergence rates of $O(1/N^{1.2716})$ across three prevalent performance measures.

Some limitations of our analysis are as follows. First, our tight analysis for the constant case can only cover the range $\alpha \in (0, \sqrt{2}]$, rather than the traditional range $(0,2)$. 
This gap might be addressed by leveraging recent advances in GD method in \cite{RGP24}, where tight bounds for the GD method with a constant stepsize across the full range $\alpha \in (0,2)$ were established. 
Second, we have maintained the proximal parameter $\lambda > 0$ as constant throughout all the iterations. This may be overly restrictive in practice.
Hence, conducting an in-depth analysis of RPPA with varying proximal parameters $\lambda_{k}$ presents an interesting avenue for future research. It remains an open task for further studies.
Moreover, there is also room for improvement in the constructed silver stepsize schedules $\overline{\pi}^{(m)}$ and $\underline{\pi}^{(m)}$, as they do not precisely match the numerical optimal stepsize schedules. 
Determining the exact optimal stepsize schedule also remains an open problem for future study.

\begin{table}[H]
\begin{minipage}[c]{0.68\textwidth}
  \centering
  {\footnotesize{}}%
  \begin{tabular}{ccccc}
  \toprule 
  \multirow{2}{2em}{ \centering{} \footnotesize{}$N$} & \multicolumn{2}{c}{{\footnotesize{}Measure: $\frac{\left\| \nabla f^{\lambda}(x^N) \right\| }{\left\| x^{0}-x^{\star} \right\|} $}} & \multicolumn{2}{c}{{\footnotesize{}Measure: $ \frac{f(z^{N}) - f(x^{\star})}{\left\| x^{0} - x^{\star} \right\|^{2}}  $}} \tabularnewline
  \cmidrule{2-3} \cmidrule{4-5}
  & { \centering{} \footnotesize{}$\pi^{(m)}$} & { \centering{} \footnotesize{}$\boldsymbol{\alpha}^{\rm gd,dist,\star}$} & { \centering{} \footnotesize{}$\overline{\pi}^{(m)}$} & { \centering{} \footnotesize{}$\boldsymbol{\alpha}^{\rm fval, dist,\star}$}\tabularnewline
  \midrule 
  {\footnotesize{}$1$} & {\footnotesize{}$ \left[\begin{array}{c}
      1.414214\end{array}\right]$ } & {\footnotesize{}$\left[\begin{array}{c}
  1.414334 \end{array}\right]$ } & {\footnotesize{}$\left[\begin{array}{c}
      1.618034\end{array}\right]$ } & {\footnotesize{}$\left[\begin{array}{c}
  1.618035 \end{array}\right]$ }\tabularnewline
  \midrule 
  {\footnotesize{}$2$} & {\footnotesize{} $-$ } & {\footnotesize{}$\left[\begin{array}{c}
  1.600519 \\
  1.414482
  \end{array}\right]$ } & {\footnotesize{}$ \left[\begin{array}{c}
      1.414214\\
      2.132242
    \end{array}\right]$ } & {\footnotesize{}$\left[\begin{array}{c}
  1.515086\\
  2.038664 \end{array}\right]$ }\tabularnewline
  \midrule 
  {\footnotesize{}$3$} & {\footnotesize{} $  \left[\begin{array}{c}
    1.414214 \\
    2.0 \\
    1.414214
    \end{array}\right] $ } & {\footnotesize{}$\left[\begin{array}{c}
    1.414745 \\
    2.000451 \\
    1.414115 
    \end{array}\right]$ } & {\footnotesize{} $-$ } & {\footnotesize{}$\left[\begin{array}{c}
    1.414211\\
    1.601232\\
    2.565298 \end{array}\right]$ }\tabularnewline
  \midrule 
  {\footnotesize{}$4$} & {\footnotesize{} $-$ } & {\footnotesize{}$\left[\begin{array}{c}
    1.414214 \\
    1.601232 \\
    2.260578 \\
    1.414214 
    \end{array}\right]$ } & {\footnotesize{} $  \left[\begin{array}{c}
    1.414214 \\
    2.0 \\
    1.414214\\
    2.965447
    \end{array}\right] $ } & {\footnotesize{}$\left[\begin{array}{c}
    1.414214\\
    2.0\\
    1.414213\\
    2.965447 \end{array}\right]$ }\tabularnewline
  \midrule 
  {\footnotesize{}$7$} & {\footnotesize{} $  \left[\begin{array}{c}
    1.414214 \\
    2.0 \\
    1.414214\\
    3.414214 \\
    1.414214 \\
    2.0 \\
    1.414214\\
    \end{array}\right]$ } & {\footnotesize{}$\left[\begin{array}{c}
    1.414402 \\
    2.000943 \\
    1.413279\\
    3.411508 \\
    1.413362 \\
    2.000698 \\
    1.414120\\
    \end{array}\right]$ } & {\footnotesize{} $-$ } & {\footnotesize{}$\left[\begin{array}{c}
    1.414214 \\
    1.601232 \\
    3.989651 \\
    1.414214 \\
    2.745299 \\
    1.515087 \\
    2.038664 \end{array}\right]$ }\tabularnewline
  \midrule 
  {\footnotesize{}$8$} & {\footnotesize{} $-$ } & {\footnotesize{}$\left[\begin{array}{c}
    1.414215 \\
    2.000003 \\
    1.414214\\
    3.754383 \\
    1.414213 \\
    1.601232 \\
    2.260578\\
    1.414214 \end{array}\right]$ } & {\footnotesize{} $ \left[\begin{array}{c}
    1.414214 \\
    2.0 \\
    1.414214\\
    3.414214 \\
    1.414214 \\
    2.0 \\
    1.414214\\
    4.284319
    \end{array}\right]$ } & {\footnotesize{}$\left[\begin{array}{c}
    1.414214 \\
    2.0 \\
    1.414214\\
    5.004706 \\
    1.414214 \\
    2.0 \\
    1.414214\\
    2.965447 \end{array}\right]$ }\tabularnewline
  \bottomrule
  \end{tabular}
  \caption{{Comparison between $\boldsymbol{\alpha} ^{\rm gd,dist,\star}$ and
$\pi^{(m)}$, as well as $\boldsymbol{\alpha} ^{\rm fval,dist,\star}$ and $\overline{\pi}^{(m)}$, where
   $\boldsymbol{\alpha}^{\rm gd, dist,\star}$ and $\boldsymbol{\alpha}^{\rm fval, dist,\star}$ denote
   the optimal stepsize schedules computed by BnB-PEP under the measures $\left\| \nabla f^{\lambda}(x^N) \right\| / \left\| x^{0}-x^{\star} \right\|$  and $(f(z^{N}) - f(x^{\star})) / \left\| x^{0} - x^{\star} \right\|^{2}$, respectively.
   Recall that $\pi^{(m)}$ is available only for $N=2^m-1$ ($m\geq 1$), and $\overline{\pi}^{(m)} $ is available only for $N=2^m$ ($m\geq 0$). \label{tab:Globally-optimal-stepsize-compare}
   }}
\end{minipage}
\hfill
\begin{minipage}[c]{0.3\textwidth}
    \centering
    \begin{tabular}{cccc}
      \toprule 
      \multirow{2}{2em}{ \centering{} \footnotesize{}$N$} & $\overline{\pi}^{(m)}$ & $\boldsymbol{\alpha} ^{\rm fval, dist,\star}$ \tabularnewline
      \cmidrule{2-3}
      & \multicolumn{2}{c}{
      \begin{tabular}[c]{c}
          \footnotesize{} Tight upper bound    \\
        \footnotesize{} of $\frac{f(z^N) - f(x^{\star})}{\left\| x^{0}-x^{\star} \right\|^{2}}$
      \end{tabular}
      } \tabularnewline
      \midrule
      {\footnotesize{}$1$} & 0.095492 & 0.095492 \tabularnewline
      \midrule
      {\footnotesize{}$2$} & \bf 0.054988 &  0.054900 \tabularnewline
      \midrule
      {\footnotesize{}$4$} & 0.028429 & 0.028429 \tabularnewline
      \midrule
      {\footnotesize{}$8$} & \bf 0.013620 & 0.013422 \tabularnewline
      \midrule
      \bottomrule
    \end{tabular}
    \caption{Comparison between the tight upper bound of $\frac{f(z^N) - f(x^{\star})}{\left\| x^{0}-x^{\star} \right\|^{2}}$ corresponding to 
 $\overline{\pi}^{(m)}$ and 
    $\boldsymbol{\alpha} ^{\rm fval, dist,\star}$
    with $m = 0, 1, 2, 3$ (corresponding to $N=1,2,4,8$, respectively). 
    In the computation, we set $\lambda=1$ for RPPA. 
    Each table entry is obtained through solving the corresponding performance estimation problem.
    We observe that  $\overline{\pi}^{(m)}$
    attains larger tight worst-case bounds than $\boldsymbol{\alpha} ^{\rm fval, dist,\star}$ when $m=1,3$ (corresponding to $N=2, 8$, respectively). 
    This may indicate that $\overline{\pi}^{(m)}$ is suboptimal. 
    \label{tab:numeric-compare}}
\end{minipage}
\end{table}

\bibliographystyle{realpha} 
\bibliography{ref.bib}

\end{document}